\documentclass[12pt, twoside]{article}
\usepackage{amsfonts}
\usepackage{xcolor}
\usepackage{mathrsfs}
\usepackage[all,cmtip]{xy}
\usepackage{amsmath}
\usepackage{amssymb,amsthm,upref,amscd}
\usepackage{setspace}
\usepackage{enumerate}
\usepackage[titletoc]{appendix}
\usepackage{times}
\usepackage{cite}
\usepackage{tikz}
\usepackage{authblk}
\usepackage[colorinlistoftodos]{todonotes}
\usetikzlibrary{arrows}
\numberwithin{equation}{section}

\def\titlerunning#1{\gdef\titrun{#1}}
\makeatletter
\def\author#1{\gdef\autrun{\def\and{\unskip, }#1}\gdef\@author{#1}}

\makeatother

\def\subjclass#1{{\renewcommand{\thefootnote}{}%
\footnote{\emph{Mathematics Subject Classification (2010):} #1}}}
\def\keywords#1{\par\medskip
\noindent\textbf{Keywords.} #1}

\allowdisplaybreaks

\theoremstyle{plain}
\newtheorem{Thm}{Theorem}[section]
\newtheorem{Lem}[Thm]{Lemma}
\newtheorem{claim}{Claim}[section]
\newtheorem{Cor}[Thm]{Corollary}
\newtheorem{Prop}[Thm]{Proposition}
\newtheorem*{Thm*}{Theorem}
\newtheorem*{claim*}{Claim}
\theoremstyle{definition}

\newtheorem*{Def*}{Definition}
\newtheorem*{Cor*}{Corollary}
\newtheorem{Rem}[Thm]{Remark}

\newcommand{\equ}{equation}
\newcommand{\C}{\mathbb{C}}
\newcommand{\N}{\mathbb{N}}
\newcommand{\R}{\mathbb{R}}
\newcommand{\Z}{\mathbb{Z}}

\DeclareMathOperator{\meas}{meas}
\DeclareMathOperator{\supp}{supp}

\DeclareMathOperator{\vol}{vol}

\newcommand\loc{\mathrm{loc}}

\let\nhatoksa=\theenumi
\let\nhatoksb=\labelenumi
\let\nhatoksc=\theenumii
\let\nhatoksd=\labelenumii
\newlength{\nhalengtha}
\newlength{\nhalengthb}
\newlength{\nhalengthc}
\newcommand{\resetenum}{
\let\theenumi=\nhatoksa
\let\labelenumi=\nhatoksb
\let\theenumii=\nhatoksc
\let\labelenumii=\nhatoksd
\setlength{\leftmargini}{\nhalengtha}
\setlength{\leftmarginii}{\nhalengthb}
\setlength{\labelwidth}{\nhalengthc}
}

\newcommand\cg{\mathcal{G}}

\newcommand\ck{\mathcal{K}}
\newcommand\cl{\mathcal{L}}

\newcommand\rr{\mathcal{R}}
\newcommand\cs{\mathcal{S}}

\newcommand\cw{\mathcal{W}}

\def\mbs{\mathbb{S}}

\def\msf{\mathscr{F}}

\def\msj{\mathscr{J}}

\def\msn{\mathscr{N}}

\def\ig{\textit{g}}

\def\ov{\overline}

\def\pa {\partial}

\def\op{\oplus}

\def\al {\alpha}
\def\bt {\beta}
\def\de {\delta}
\def\Ga {\Gamma}

\def\lm {\lambda}

\def\om {\omega}
\def\Om {\Omega}
\def\sa {\sigma}
\def\vr {\varepsilon}
\def\va {\varphi}

\def\span{\hbox{span}}

\def\meas{\hbox{\it meas}}

\def\vol{\mathrm{vol}}
\def\real{\hbox{Re}}

\newcommand{\inp}[2]{\left\langle#1,#2\right\rangle}

\begin{document}

\baselineskip=17pt

\titlerunning{Blow-up profile for spinorial Yamabe type equation on  $S^m$}

\title{   \vspace{-3em}
Blow-up profile for spinorial Yamabe type equation on  $S^m$}

\author{Tian Xu\footnote{\noindent
 Supported by the National Science Foundation of China (NSFC 11601370, 11771325)
and the Alexander von
\newline \hspace*{1.5em}
Humboldt Foundation of Germany}  \vspace{-2em}
}

\date{}

\maketitle


\begin{abstract}

Motivated by recent progress on a spinorial analogue of the Yamabe problem in the geometric literature, we study a conformally invariant spinor field equation on the $m$-sphere, $m\geq2$. Via variational methods, we study analytic aspects of the associated energy functional, culminating in a blow-up analysis.


\end{abstract}


\section{Introduction}

Within the framework of Spin Geometry, a problem analogous to the Yamabe problem has received increasing attention in recent years.
Several works of Ammann \cite{Ammann, Ammann2009} and Ammann, Humbert and others \cite{AGHM, AHA, AHM} provide a brief picture of how variational method may be employed to the investigation. Their starting point was the Hijazi inequality \cite{Hij86} which links the first eigenvalue of two important elliptic differential operators: the conformal Laplacian and the Dirac operator.

Let $(M,\ig,\sa)$ be an $m$-dimensional closed spin manifold with a
metric $\ig$, a spin structure $\sa:P_{Spin}(M)\to P_{SO}(M)$ and
a spin representation $\rho: Spin(m)\to End(\mbs_m)$.
Let us denote by $\mbs(M)=P_{Spin}(M)\times_{\rho}\mbs_m$ the spinor bundle on $M$ and $D_\ig: C^\infty(M,\mbs(M))\to C^\infty(M,\mbs(M))$ the Dirac operator (see Section \ref{preliminaries} for more details). A spin conformal invariant is defined as
\begin{\equ}\label{BHL}
\lm_{min}^+(M,[\ig],\sa):=\inf_{\tilde\ig\in[\ig]}\lm_1^+(\tilde\ig)
\text{Vol}(M,\tilde\ig)^{\frac1m}
\end{\equ}
where $\lm_1^+(\tilde\ig)$ denotes the smallest positive eigenvalue of the Dirac operator $D_{\tilde\ig}$ with respect to the conformal metric $\tilde\ig\in[\ig]:=\big\{ f^2\ig:\, f\in C^{\infty}(M),\, f>0 \big\}$. Ammann points out in \cite{Ammann, Ammann2009} that studying critical metrics for this invariant involves similar analytic problems to those appearing in the Yamabe problem. It follows that finding a critical metric of \eqref{BHL} is  equivalent to prove the
existence of a spinor field $\psi\in C^{\infty}(M,\mbs(M))$ minimizing the functional defined by
\begin{\equ}\label{J-ig-functional}
J_\ig(\phi)=\frac{\Big(
\int_M|D_{\ig}\phi|^{\frac{2m}{m+1}}d\vol_{\ig}
\Big)^{\frac{m+1}{m}}}{\big|\int_M(D_{\ig}\phi,\phi)d\vol_{\ig}\big|}
\end{\equ}
with the Euler-Lagrange equation
\begin{\equ}\label{Dirac-c}
D_\ig\psi=\lm_{min}^+(M,[\ig],\sa) |\psi|^{\frac2{m-1}}\psi.
\end{\equ}

As was pointed out in \cite{Ammann}, standard variational method does not imply the existence of minimizers for $J_\ig$ directly. This is due to the  criticality of the non-quadratic term in \eqref{J-ig-functional}. Indeed, the exponent $\frac{2m}{m+1}$ is critical in the sense that Sobolev embedding involved is precisely the one for which the compactness is lost in the Reillich-Kondrakov theorem. Similar to the argument in solving the Yamabe problem, one might be able to find a criterion which recovers the compactness. And it is crucial to note that a spinorial analogue of Aubin's inequality holds (see \cite{AGHM})
\begin{\equ}\label{spinorial Aubin inequ}
\lm_{min}^+(M,[\ig],\sa)\leq \lm_{min}^+(S^m,[\ig_{S^m}],\sa_{S^m})=\frac m2\om_m^{\frac1m}
\end{\equ}
where $(S^m,\ig_{S^m},\sa_{S^m})$ is the $m$-dimensional sphere equipped with its canonical metric $\ig_{S^m}$ and its standard spin structure $\sa_{S^m}$, and $\om_m$ is the standard volume of $(S^m,\ig_{S^m})$. The criterion obtained in \cite{Ammann} shows that if inequality \eqref{spinorial Aubin inequ} is strict then the spinorial Yamabe problem \eqref{Dirac-c} has a nontrivial solution minimizing the functional $J_\ig$.

Tightly related to geometric data, the nonlinear problem \eqref{Dirac-c} provides a strong tool for showing the existence of constant mean curvature hypersurfaces in Euclidean spaces. This is one of the most attractive features of the spinorial Yamabe problem that unseals new researches in both PDE theory and Riemann geometry. In this paper, we are concerned with a more general form of \eqref{Dirac-c}:
\begin{\equ}\label{Dirac-general}
D_\ig\psi=H(\xi)|\psi|^{\frac2{m-1}}\psi \quad \text{on } M
\end{\equ}
where $H:M\to\R$ is a given function on $M$. As was observed in \cite{Friedrich, KS, Taimanov1, Taimanov2, Taimanov3}, the function
$H$ plays the role of the mean curvature for a conformal immersion $M\to\R^{m+1}$. Indeed, if there exists a conformal immersion of $M$ into $\R^{m+1}$ with mean curvature $H$, then from the restriction to $M$ of any parallel spinor on $\R^{m+1}$, one obtains a spinor field which satisfies the equation \eqref{Dirac-general}. The converse is also true if $m=2$. Particularly, if $M$ is a simply connected surface, i.e. a $2$-sphere, then for any non-trivial solution to \eqref{Dirac-general} there corresponds to a conformal immersion (possibly with branching points) $S^2\to\R^3$ with mean curvature $H$. This one-to-one  correspondence is the so-called \textit{Spinorial Weierstra\ss\ representation}, see \cite{Friedrich} for a nice explanation and \cite{KS} for a rich source of results on Dirac operators on compact surfaces.

Raulot \cite{Raulot} obtained an existence criterion for the problem \eqref{Dirac-general} which is similar to the Aubin's inequality. One of his results shows that if the Dirac operator $D_\ig$ is invertible, $H$ is positive and satisfies
\begin{\equ}\label{criterion}
\lm_{min,H}\cdot\big(\max_{\xi\in M} H\big)^{\frac{m-1}{m}}<\frac m2\om_m^{\frac1m},
 \quad \lm_{min,H}:=\inf_{\phi\neq0}\frac{\Big(
\int_M H^{-\frac{m-1}{m+1}}|D_{\ig}\phi|^{\frac{2m}{m+1}}d\vol_{\ig}
\Big)^{\frac{m+1}{m}}}{\big|\int_M(D_{\ig}\phi,\phi)d\vol_{\ig}\big|}
\end{\equ}
then there exists a solution to the equation
\[
D_\ig\psi=\lm_{min,H} H(\xi)|\psi|^{\frac2{m-1}}\psi
\quad \text{and} \quad
\int_MH(\xi)|\psi|^{\frac{2m}{m-1}}d\vol_\ig=1.
\]
However, condition \eqref{criterion} is only verified for some special cases and general existence result is still lacking (cf. \cite{AHM, Ginoux }). In particular, the existence results in \cite{Raulot} does not apply for $M=S^m$ since the strict inequality of \eqref{criterion} is not valid on the spheres in any circumstance.

A different point of view was introduced by Isobe, who suggested to consider the geometric property of the function $H$ for the spinorial Yamabe equation \eqref{Dirac-general}. In his paper \cite{Isobe-MathAnn}, Isobe established existence results for $H$ close to a constant function on $S^m$. The idea was to count Morse index at critical points of $H$ and to pose an index counting condition. Similar idea has been employed in the study of scalar curvature equations, see for instance Bahri and Coron \cite{BahriC} and Chang et al. \cite{CGY, CY}.

In this paper, we will also consider the spinorial Yamabe equation \eqref{Dirac-general} on the $m$-sphere, i.e. we take $M=S^m$, $m\geq2$.
We aim to provide an analytic foundation for Eq. \eqref{Dirac-general}
with general functions $H$ (the case $H$ is "far" from a constant function is allowed). Inspired by Yamabe \cite{Yamabe} and Aubin \cite{Aubin}, we start with basic points like the compactness of subcritical approximation argument. As was established in the Yamabe problem, the key analytical points are the singularities in solutions of subcritical approximation equations which can appear at isolated points. For Eq. \eqref{Dirac-general}, at those isolated singularities, rescaling produces an entire solution of the spinorial Yamabe equation in the Euclidean space which can be compactified to a spherical "bubble" . These entire solutions or "bubbles" can be viewed as obstructions to the compactness of Eq. \eqref{Dirac-general}. Hence, another important point will be the asymptotic behavior of such entire solutions. In this paper, we therefore perform such an analysis for the spinorial Yamabe equation. As in the classical cases, this provides a basic analytical picture. A general existence result and the geometric applications of Eq. \eqref{Dirac-general} will be established in a forthcoming paper.

The subcritical approximate point of view is to consider the problem of the form
\begin{\equ}\label{Dirac-eps}
D_\ig\psi=H(\xi)|\psi|^{\frac2{m-1}-\vr}\psi \quad \text{on } M
\end{\equ}
where $\vr>0$ is small. The above equation has a variational structure. In fact, $\psi\in C^1(M,\mbs(M))$ is a solution to Eq. \eqref{Dirac-eps} if and only if $\psi$ is a critical point of the functional
\begin{\equ}\label{functional-eps}
\cl_\vr(\psi):=\frac12\int_M(D_\ig\psi,\psi)d\vol_{\ig}
-\frac1{2^*-\vr}\int_M H(\xi)|\psi|^{2^*-\vr}
d\vol_{\ig}
\end{\equ}
where $2^*:=\frac{2m}{m-1}$. Assume that $\{\psi_\vr\}$ is a sequence
of solutions of \eqref{Dirac-eps} with
$\cl_\vr(\psi_\vr)\leq C_0$
for some positive constant $C_0$, Raulot's condition \eqref{criterion} can be interpreted as (via the dual variational principle) if $C_0<C_{crit}:=\frac1{2m(\max H)^{m-1}}\big(\frac m2\big)^m\om_m$ then $\{\psi_\vr\}$ admits a subsequence converging to a smooth solution of \eqref{Dirac-general}.

When $C_0$ is big (i.e. the condition \eqref{criterion} fails), the singularities or the so-called "blow-up" phenomenon may occur. To simplify the presentations of our main results, we set $M=S^m$ in \eqref{Dirac-eps} and \eqref{functional-eps} and denote $Crit[H]=\{\xi\in S^m:\, \nabla H(\xi)=0\}$ for a fixed function $H\in C^1(S^m)$ and $H>0$. Assume that $\{\psi_\vr\}$ is a sequence of solutions of \eqref{Dirac-eps} and satisfying
\begin{\equ}\label{energy-condition}
C_{crit} \leq \cl_\vr(\psi_\vr) \leq 2C_{crit}-\theta
\end{\equ}
for some small constant $\theta>0$. Define
\[
\Sigma=\big\{a\in S^m:\, \text{there exists } a_\vr\to a \text{ such that } |\psi_\vr(a_\vr)|\to+\infty \text{ as } \vr\to0\big\}.
 \]
Then we show that $\Sigma\subset Crit[H]$ and $\{\psi_\vr\}$ admits a subsequence, still denoted by $\{\psi_\vr\}$, satisfying one of the following alternatives
\begin{itemize}
\item[(1)] (compactness) $\Sigma=\emptyset$, $\{\psi_\vr\}$ is compact in $C^1(S^m,\mbs(S^m))$ and converges to a solution $\psi$ of Eq. \eqref{Dirac-general} as $\vr\to0$;

\item[(2)] (concentration) $\Sigma\neq \emptyset$ contains a single point and $|\psi_\vr|\to0$ as $\vr\to0$ uniformly on compact subsets of $S^m\setminus \Sigma$.
\end{itemize}
The conclusion described above is only a rough summation of the results in Proposition \ref{alternative prop} and \ref{blow-up prop}\,-\,\ref{blow-up cor}.

Observe that condition \eqref{energy-condition} performs a key ingredient in the blow-up profile, as it exhibits a unique
microscopic bubbling pattern. By virtue of Raulot's result, it can be seen that $C_{crit}$ is the threshold for bubbling. The alternative blow-up profile obtained in the present paper can be viewed as a local version of concentration-compactness type result above the bubbling threshold, i.e. in the range $[C_{crit},2C_{crit})$.
A geometric motivation lying behind our interest in this $2C_{crit}$ upper bound is, in dimension 2, it would give an upper bound estimate for the Willmore energy provided that one can rule out the formation of bubbles. Indeed,
in the setting of Spinorial Weierstra\ss\ representation, if $\psi\in C^1(S^2,\mbs(S^2))$ is a nontrivial solution of \eqref{Dirac-general}
then $(S^2, |\psi|^4\ig_{S^2})$ can be isometrically immersed into $\R^3$ with mean curvature being prescribed as $H$. Then the Willmore energy of such an immersion is defined by
\[
\cw(\psi)=\int_{S^2}H^2(\xi)d\vol_{|\psi|^4\ig_{S^2}}
=\int_{S^2}H^2(\xi)|\psi|^4d\vol_{\ig_{S^2}}.
\]
Once the blow-up is ruled out, \eqref{energy-condition} and the compactness of the subcritical approximating sequence $\{\psi_\vr\}$ gives a solution $\psi$ to Eq. \eqref{Dirac-general} such that
\[
\cw(\psi)\leq \max_{\xi\in S^2} H\int_{S^2}H(\xi)|\psi|^4d\vol_{\ig_{S^2}}
=\max_{\xi\in S^2} H\cdot\lim_{\vr\to0}\Big(4\cl_\vr(\psi_\vr)-2\cl_\vr'(\psi_\vr)[\psi_\vr]\Big)
< 8 \pi.
\]
Therefore, by Li and Yau's inequality \cite[Theorem 6]{LY}, we obtain the immersion of $(S^2,|\psi|^4\ig_{S^2})$ into $\R^3$ has no self-intersection, i.e. it is a global embedding. This will be of particular geometric interests and will appear in our future attempt on the existence issues.

The paper is organized as follows: Section \ref{preliminaries} contains the basic geometric backgrounds and functional settings; in Section \ref{perturbation sec} the continuity of the energy functionals with respect to the subcritical approximation is established; Section \ref{Blow-up section} contains the detailed proofs of the alternative blow-up profile.


\section{Preliminaries}\label{preliminaries}
\subsection{Spin structure and the Dirac operator}
Let $(M,\ig)$ be an $m$-dimensional Riemannian manifold
with a chosen orientation. Let $P_{SO}(M)$ be the set of
positively oriented orthonormal frames on $(M,\ig)$. This
is a $SO(m)$-principal bundle over $M$. A {\it spin structure}
on $M$ is a pair $\sa=(P_{Spin}(M),\vartheta)$ where $P_{Spin}(M)$
is a $Spin(m)$-principal bundle over $M$ and
$\vartheta: P_{Spin}(M)\to P_{SO}(M)$ is a map such that
\begin{displaymath}
\xymatrix@R=0.3cm{
P_{Spin}(M)\times Spin(m)  \ar[r]
\ar[dd]^{\displaystyle\vartheta\times \Theta}&
P_{Spin}(M)\ar[dd]^{\displaystyle\vartheta} \ar[dr] \\
&&  M \\
P_{SO}(M)\times SO(m) \ar[r] & P_{SO}(M)  \ar[ur] }
\end{displaymath}
commutes, where $\Theta: Spin(m)\to SO(m)$ is the
nontrivial double covering of $SO(m)$. There is a
topological condition for the existence of a spin
structure, that is, the vanishing of the second Stiefel-Whitney
class $\om_2(M)\in H^2(M,\Z_2)$.
Furthermore, if a spin structure exists, it need not to be
unique. For these, we refer to \cite{Friedrich, Lawson}.

To introduce the spinor bundle, we recall that the
Clifford algebra $Cl(\R^m)$ is the associative $\R$-algebra
with unit, generated by $\R^m$ satisfying the relation
$x\cdot y-y\cdot x=-2(x,y)$ for $x,y\in\R^m$ (here
$(\cdot,\cdot)$ is the Euclidean scalar product on
$\R^m$). It turns out that $Cl(\R^m)$ has a smallest
representation $\rho: Spin(m)\to End(\mbs_m)$
of dimension $\dim_\C(\mbs_m)=2^{[\frac{m}2]}$ such
that $\C l(\R^m):=Cl(\R^m)\otimes\C\cong End(\mbs_m)$
as $\C$-algebra. The spinor bundle is then defined as the
associated vector bundle
\[
\mbs(M):=P_{Spin}(M)\times_\rho\mbs_m.
\]
Note that the spinor bundle carries a natural Clifford
multiplication, a natural hermitian metric and a metric
connection induced from the Levi-Civita connection on
$TM$ (see \cite{Friedrich, Lawson}), this bundle
satisfies the axioms of Dirac bundle in the sense that
\begin{itemize}
\item[$(i)$] for any $x\in M$, $X,Y\in T_xM$ and
$\psi\in\mbs_x(M)$
\[
X\cdot Y\cdot \psi + Y\cdot X\cdot\psi
+2\ig(X,Y)\psi=0;
\]
\item[$(ii)$] for any $X\in T_xM$ and $\psi_1,\psi_2\in\mbs_x(M)$,
\[
(X\cdot \psi_1,\psi_2)=-(\psi_1,X\cdot\psi_2),
\]
where $(\cdot,\cdot)$
is the hermitian metric on $\mbs(M)$;
\item[$(iii)$] for any $X,Y\in\Ga(TM)$ and $\psi\in\Ga(\mbs(M))$,
\[
\nabla_X^\mbs(Y\cdot\psi)=(\nabla_XY)\cdot\psi+Y\cdot\nabla_X^\mbs\psi,
\]
where $\nabla^\mbs$ is the metric connection on $\mbs(M)$.
\end{itemize}
On the spinor bundle $\mbs(M)$, the Dirac operator
is then defined as the composition
\begin{displaymath}
\xymatrixcolsep{1.6pc}\xymatrix{
D:\Ga(\mbs(M)) \ar[r]^-{\nabla^\mbs} &
\Ga(T^*M\otimes \mbs(M)) \ar[r] &
\Ga(TM\otimes \mbs(M)) \ar[r]^-{\mathfrak{m}} &
\Ga(\mbs(M))}
\end{displaymath}
where $\mathfrak{m}$ denotes the Clifford multiplication
$\mathfrak{m}: X\otimes\psi\mapsto X\cdot\psi$.

The Dirac operator behaves very nicely under conformal
changes in the following sense (see \cite{Hij86, Hit74}):
\begin{Prop}\label{conformal formula}
Let $\ig_0$ and $\ig=f^2\ig_0$ be two conformal metrics on
a Riemannian spin manifold $M$. Then, there exists an
isomorphism of vector bundles $\iota:\, \mbs(M,\ig_0)\to
\mbs(M,\ig)$ which is a fiberwise isometry such that
\[
D_\ig\big( \iota(\psi) \big)=\iota\big( f^{-\frac{m+1}2}D_{\ig_0}
\big( f^{\frac{m-1}2}\psi \big)\big),
\]
where $\mbs(M,\ig_0)$ and $\mbs(M,\ig)$ are spinor
bundles on $M$ with respect to the metrics $\ig_0$ and
$\ig$, respectively, and $D_{\ig_0}$ and $D_\ig$ are
the associated Dirac operators.
\end{Prop}

\subsection{Bourguignon-Gauduchon trivialization}\label{BG sec}
In what follows, let us introduce briefly a local
trivialization of the spinor bundle $\mbs(M)$ constructed by
Bourguignon and Gauduchon \cite{BG}.

Let $a\in V\subset M$ be an arbitrary point and let
$(x_1,\dots,x_m)$
be the normal coordinates given by the exponential map
$\exp_a: U\subset T_aM\cong\R^m\to V$,
$x=(x_1,\dots, x_n)\mapsto y=\exp_ax$.
Then, we have
$(\exp_a)_{*y}: (T_{\exp_a^{-1}y}U\cong\R^m,\exp_a^*\ig_y)
\to(T_yM,\ig_y)$ is an isometry for each $y\in V$.
Thus, we can obtain an isomorphism of
$SO(m)$-principal bundles:
\begin{displaymath}
\xymatrix{
P_{SO}(U,\exp_a^*\ig)
\ar[rr]^{\quad\displaystyle(\exp_a)_*}  \ar[d] & &
P_{SO}(V,\ig) \ar[d]\\
U\subset T_aM \ar[rr]^{\ \ \displaystyle\exp_a}
& & V\subset M}
\end{displaymath}
where $(\exp_a)_*(\{z_1,\dots,z_m\})
=\big\{(\exp_a)_*z_1,\dots,(\exp_a)_*z_m\big\}$ for an
oriented frame $\{z_1,\dots,z_m\}$ on $U$.
Notice that $(\exp_a)_*$ commutes with the
right action of $SO(m)$, then we
infer that $(\exp_a)_*$ induces an isomorphism of spin
 structures:
\begin{displaymath}
\xymatrix{
Spin(m)\times U=P_{Spin}(U,\ig_{\R^m})
\ar[rr]^{\quad\displaystyle\ov{(\exp_a)}_*}  \ar[d]
& & P_{Spin}(V,\ig)
\subset P_{Spin}(M) \ar[d]\\
U\subset T_aM \ar[rr]^{\ \ \displaystyle\exp_a} &  &
V\subset M}
\end{displaymath}
Hence, we can obtain an isomorphism between the spinor
bundles $\mbs(U)$ and $\mbs(V)$ by
\[
\aligned
\mbs(U):=P_{Spin}(U,\ig_{\R^m})\times_\rho \mbs_m&\to
\mbs(V):=P_{Spin}(V,\ig)\times_\rho \mbs_m \subset
\mbs(M)\\
\psi=[s,\va]&\mapsto\bar\psi:=\big[\ov{(\exp_a)}_*(s),\va\big]
\endaligned
\]
where $(\rho,\mbs_m)$ is the complex spinor representation,
and $[s,\va]$ and
$\big[\ov{(\exp_a)}_*(s),\va\big]$ denote the
equivalence classes of
$(s,\va)\in P_{Spin}(U,\ig_{\R^m})\times \mbs_m$
and $\big(\ov{(\exp_a)}_*(s),\va\big)\in P_{Spin}(V,\ig)
\times \mbs_m$ under the action of $Spin(m)$, respectively.

For more background material on the
Bourguignon-Gauduchon trivialization we refer the reader to
\cite{AGHM, BG}

\subsection{$H^{\frac12}$-spinors on $S^m$}

Let us consider the case $M=S^m$ with the standard
metric $\ig_{S^m}$. Let $Spec(D)$ denote the spectrum
of the Dirac operator $D$. It
is well-known that $D$ is essentially self-adjoint in
$L^2(S^m,\mbs(S^m))$
and has compact resolvents
(see \cite{Friedrich, Lawson, Ginoux}).
Particularly, we have
\[
Spec(D)=\Big\{\pm\big(\frac{m}2+j\big):\, j=0,1,2,\dots\Big\}
\]
and, for each $j$, the eigenvalues $\pm(\frac{m}2+j)$ have the
multiplicity
\[
2^{[\frac{m}2]}
\begin{pmatrix}
n+j-1\\
j
\end{pmatrix}
\]
(see \cite{Sulanke}). For notation convenience, we can write
$Spec(D)=\{\lm_k\}_{k\in\Z\setminus\{0\}}$ with
\[
\lm_k=\frac{k}{|k|}\big(\frac{m}2+|k|-1\big).
\]
The eigenspaces of
$D$ form a complete orthonormal decomposition of
$L^{2}(S^m,\mbs(S^m))$, that is,
\[
L^{2}(S^m,\mbs(S^m))=\ov{\bigoplus_{\lm\in Spec(D)} \ker(D-\lm I)}.
\]

Now let us denote by $\{\eta_k\}_{k\in\Z\setminus\{0\}}$ the complete
orthonormal basis of $L^2(S^m,\mbs(S^m))$ consisting of
the smooth eigenspinors of the Dirac operator $D$, i.e.,
$D\eta_k=\lm_k\eta_k$. We then define the operator
$|D|^{\frac12}: L^2(S^m,\mbs(S^m))\to L^2(S^m,\mbs(S^m))$ by
\[
|D|^{\frac12}\psi=\sum_{k\in\Z\setminus\{0\}}|\lm_k|^{\frac12}
a_k\eta_k,
\]
where $\psi=\sum_{k\in\Z\setminus\{0\}}
a_k\eta_k\in L^2(S^m,\mbs(S^m))$.
Let us set
\[
H^{\frac12}(S^m,\mbs(S^m)):=\Big\{\psi=\sum_{k\in\Z} a_k\eta_k
\in L^2(S^m,\mbs(S^m)):\, \sum_{k\in\Z
\setminus\{0\}}|\lm_k||a_k|^2<\infty\Big\}.
\]
We have $H^{\frac12}(S^m,\mbs(S^m))$ coincides with the
Sobolev space $W^{\frac12,2}(S^m,\mbs(S^m))$
(see \cite{Adams, Ammann}).
We could now endow $H^{\frac12}(S^m,\mbs(S^m))$
with the inner product
\[
\inp{\psi}{\va}=\real
\big(|D|^{\frac12}\psi,|D|^{\frac12}\va\big)_2
\]
and the induced norm $\|\cdot\|
=\inp{\cdot}{\cdot}^{\frac12}$,
where $(\psi,\va)_2=\int_{S^m}
(\psi,\va)d\vol_{\ig_{S^m}}$ is the
$L^2$-inner product on spinors.
In particular, $E:=H^{\frac12}(S^m,\mbs(S^m))$
induces a splitting $E=E^+\op  E^-$ with
\begin{\equ}\label{decomposition1}
E^+:=\ov{\span\{\eta_k\}_{k>0}} \quad
\text{and} \quad
E^-:=\ov{\span\{\eta_k\}_{k<0}},
\end{\equ}
where the closure is taken in the $\|\cdot\|$-topology.
It is then clear that these are orthogonal subspaces
of $E$ on which the action
\[
\int_{S^m}(D\psi,\psi)d\vol_{\ig_{S^m}}
\]
is positive or negative.
In the sequel, with respect to this decomposition,
we will write $\psi=\psi^++\psi^-$ for any
$\psi\in E$.  The dual space of $E$ will be denoted by
$E^*:=H^{-\frac12}(S^m,\mbs(S^m))$.


\section{Perturbation from subcritical}\label{perturbation sec}
For $p\in(2,2^*]$, let us consider
\begin{\equ}\label{perturbed problem}
D\psi=H(\xi)|\psi|^{p-2}\psi  \quad \text{on } S^m
\end{\equ}
where $H\in C^1(S^m)$ and $H>0$.
The corresponding energy functional will be
\[
\aligned
\cl_p(\psi)&=\frac12\int_{S^m}(D\psi,\psi)
d\vol_{\ig_{S^m}}-\frac1{p}
\int_{S^m} H(\xi)|\psi|^{p}d\vol_{\ig_{S^m}}  \\
&=\frac12\big(\|\psi^+\|^2-\|\psi^-\|^2\big) - \frac1{p}
\int_{S^m} H(\xi)|\psi|^{p}d\vol_{\ig_{S^m}} .
\endaligned
\]
For $p=2^*$ we will simply use the notation $\cl$
instead of $\cl_{2^*}$.
For $p<2^*$, the problem is subcritical and, due to the
compact embedding of $E\hookrightarrow
L^p(S^m,\mbs(S^m))$, it is not
difficult to see that there always exist nontrivial weak solutions
of \eqref{perturbed problem}. So,
it remains to find conditions under which solutions
of the subcritical problem will converge to a nontrivial solution
of the critical problem.
In what follows, we will first study the behavior of the energy
functional $\cl_p$, $p\in(2,2^*]$.

For notation convenience, in the sequel,
by $L^p$ we denote the Banach
space $L^p(S^m,\mbs(S^m))$ for $p\geq1$
and by $|\cdot|_p$ we denote
the usual $L^p$-norm. We also denote $H_{max}=\max_{S^m} H$
and $H_{min}=\min_{S^m} H$.
And without loss of generality, we assume
$\int_{S^m}H(\xi)d\vol_{\ig_{S^m}}=1$.

\subsection{A reduction argument}\label{reduction sec}

Given $u\in E^+\setminus\{0\}$, we set
\[
W(u)=\span\{u\}\op E^-=\big\{
\psi\in E:\, \psi=tu+v,\, v\in  E^-,\, t\in\R \big\}.
\]

\begin{Lem}\label{anti-coercive}
For $u\in E^+\setminus\{0\}$, $\cl_p$
is anti-coercive on $W(u)$, that is,
\[
\cl_p(\psi)\to-\infty \text{ as } \|\psi\|\to\infty, \ \psi\in W(u).
\]
\end{Lem}
\begin{proof}
To begin with, for $\psi\in W(u)$, let us write $\psi=tu+v$
with $v\in E^-$. We may then fix a constant $C_p>0$ such that
\[
|\psi|_p\geq C_p\big(|tu|_p+|v|_p\big)
\text{ for all } \psi\in W(u).
\]
Therefore, we can infer
\[
\aligned
\cl_p(\psi)&\leq \frac12\big( \|\psi^+\|^2-\|\psi^-\|^2\big)
-\frac{H_{min}}p|\psi|_p^p \\
&\leq \frac{t^2}2\Big( \|u\|^2-\frac{2H_{min}}p\cdot C\cdot t^{p-2}
|u|_p^p \Big)-\frac12\|v\|^2.
\endaligned
\]
And thus the conclusion follows.
\end{proof}

For a fixed $u\in E^+\setminus\{0\}$, we define
$\phi_u:  E^-\to\R$ by
\[
\phi_u(v):=\cl_p(u+v)=\frac12\int_{S^m}
(D(u+v),u+v)d\vol_{\ig_{S^m}}
-\frac1p\int_{S^m} H(x)|u+v|^pd\vol_{\ig_{S^m}}.
\]
From the convexity of the map
$\psi\mapsto \int_{S^m} H(x)|\psi|^pd\vol_{\ig_{S^m}}$,
it can be straightforwardly verified that
\begin{\equ}\label{concave}
\phi_u''(v)[w,w]\leq\int_{S^m}(Dw,w)d\vol_{\ig_{S^m}}
-\int_{S^m} H(x)|u+v|^{p-2}
\cdot|w|^2 d\vol_{\ig_{S^m}}\leq-\|w\|^2
\end{\equ}
for all $v,w\in  E^-$. This suggests that $\phi_u$ is concave.
Moreover, we have

\begin{Prop}\label{reduction}
There exists a $C^1$ map $h_p: E^+\to E^-$ such that
\[
\|h_p(u)\|^2\leq\frac2p\int_{S^m}H(x)|u|^p
d\vol_{\ig_{S^m}}
\]
and
\begin{\equ}\label{hlm}
\cl_p'(u+h_p(u))[v]=0 \quad \forall v\in E^-.
\end{\equ}
Furthermore, if denoted by
$I_p(u)=\cl_p(u+h_p(u))$, the function
$t\mapsto I_p(tu)$ is $C^2$ and, for
$u\in E^+\setminus\{0\}$ and $t>0$,
\begin{\equ}\label{it}
I'_p(tu)[u]=0 \,\Rightarrow\, I_p''(tu)[u,u]<0.
\end{\equ}
\end{Prop}
\begin{proof}
We sketch the proof as follows.
First of all, by Lemma \ref{anti-coercive}, we have
$\lim_{\|v\|\to\infty}\phi_u(v)=-\infty$ which implies
$\phi_u$ is anti-coercive. Then it follows from
\eqref{concave} and the weak sequential upper
semi-continuity of $\phi_u$ that there exists a unique strict
 maximum point $h_p(u)$ for $\phi_u$,
which is also the only critical point of $\phi_u$ on
$E_\lm^-$.

Notice that
\[
\aligned
0&\leq\cl_p(u+h_p(u))-\cl_p(u)\\
&=\frac12\big( \|u\|^2-\|h_p(u)\|^2 \big)
-\frac1p\int_{S^m}H(\xi)|u+h_p(u)|^pd\vol_{\ig_{S^m}}
-\frac12\|u\|^2\\
&\qquad+\frac1p\int_{S^m}H(\xi)|u|^p
d\vol_{\ig_{S^m}},
\endaligned
\]
we therefore have
\[
\|h_p(u)\|^2\leq\frac2p\int_{S^m}H(x)|u|^p
d\vol_{\ig_{S^m}}.
\]

We define $\bt: E^+\times E^-\to E^-$ by
\[
\bt(u,v)=\phi_u'(v)=\cl_p'(u+v)\big|_{E^-},
\]
where we have identified $E^-$ with its dual space.
Observe that, for every $u\in E^+$, we have
\[
\phi_u'(h_p(u))[w]=\cl_p'(u+h_p(u))[w]=0, \quad
\forall w\in E^-.
\]
This implies $\bt(u,h_p(u))=0$ for all $u\in E^+$.
Notice that $\pa_v\bt(u,h_p(u))=\phi_u''(h_p(u))$ is a
bilinear form on $E^-$ which is bounded and anti-coercive.
Hence $\pa_v\bt(u,h_p(u))$ is an isomorphism. And
therefore, by the implicit function theorem, we can infer
that the uniquely determined map $h_p: E^+\to E^-$
is of $C^1$ smooth with its derivative given by
\begin{\equ}\label{h-derivative}
h_p'(u)=-\pa_v\bt(u,h_p(u))^{-1}\circ\pa_u\bt(u,h_p(u)),
\end{\equ}
this completes the proof of the first statement.

To prove \eqref{it}, it sufficient to show that
\begin{\equ}\label{I-equivalent}
\textit{If } u\in E^+\setminus\{0\} \textit{ satisfies }
I_p'(u)[u]=0, \textit{ then } I_p''(u)[u,u]<0.
\end{\equ}
Now, for simplicity,  let us denote
$\cg_p: E\to\R$ by $\cg_p(\psi)=\frac1p
\int_{S^m}H(\xi)|\psi|^pd\vol_{\ig_{S^m}}$
and set $z=u+h_p(u)$ and $w=h_p'(u)[u]-h_p(u)$.
By using $\cl_p'(u+h_p(u))|_{E^-}\equiv0$, we have
\eqref{I-equivalent} is a consequence of the following computation
\begin{eqnarray}\label{eee1}
I_p''(u)[u,u]&=&\cl_p''(z)[u+h_p'(u)[u],u]=\cl_p''(z)[z+w,z+w]  \nonumber\\
&=&\cl_p''(z)[z,z]+2\cl_p''(z)[z,w]
+\cl_p''(z)[w,w]  \nonumber\\
&=&I_p'(u)[u]+\big(\cg_p'(z)[z]-\cg_p''(z)[z,z]\big)
+2\big(\cg_p'(z)[w]-\cg_p''(z)[z,w] \big) \nonumber\\
& &\qquad -\cg_p''(z)[w,w]
+\int_{S^m}(Dw,w)d\vol_{\ig_{S^m}} \nonumber\\
& \leq& I_p'(u)[u]-\frac{p-2}{p-1}\int_{S^m}H(\xi)|z|^p
 d\vol_{\ig_{S^m}}-\|w\|^2.
\end{eqnarray}
%
%
\end{proof}

In the sequel, we shall call $(h_p,I_p)$
the reduction couple for $\cl_p$ on $E^+$.
It is all clear that critical points of $I_p$ and $\cl_p$ are
in one-to-one correspondence via the injective map
$u\mapsto u+h_p(u)$ from $E^+$ to $E$.
Let us set
\begin{\equ}\label{nehari}
\msn_p=\big\{u\in E^+\setminus\{0\}:\,
I_p'(u)[u]=0 \big\}.
\end{\equ}
By Proposition \ref{reduction}, we have $\msn_p$
is a smooth manifold of codimension $1$ in $E^+$
and is a natural constraint for the problem of finding non-trivial
critical points of $I_p$.
Furthermore, the function $t\mapsto I_p(tu)$
attains its unique critical point at $t=  t(u)>0$ (such that
$t(u)u\in\msn_p$) which realizes its maximum. And
at this point, we have
\[
\max_{t>0}I_p(tu)=\max_{\psi\in W(u)}\cl_p(\psi).
\]

\subsection{Continuity with respect to the perturbation}
Now we will work with a convergent sequence
$\{p_n\}$ in $(2,2^*]$ and we allow the case
$p_n\equiv p$. The main result here
is the following:
\begin{Prop}\label{continuity prop}
Let $p_n\to p$ in $(2,2^*]$ as $n\to\infty$ and
$c_1,c_2>0$.
For any $\theta>0$, there exists $\al>0$ such
that for all large $n$ and $\psi\in E$ satisfying
\[
c_1\leq\cl_{p_n}(\psi)\leq c_2 \quad \text{and} \quad
\|\cl_{p_n}'(\psi)\|_{E^*}\leq\al
\]
we have
\[
\max_{t>0}I_{p_n}(t\psi^+)\leq \cl_{p_n}(\psi)+\theta.
\]
\end{Prop}

This result provides a kind of continuity of energy levels
with respect to the exponents $p_n$. This is due to the
fact, by Proposition \ref{reduction}, we always have
$\cl_{p_n}(\psi)\leq\max_{t>0}I_{p_n}(t\psi^+)$.

The proof of Proposition \ref{continuity prop} is divided
into several steps. First of all, let $p_n\to p$ in $(2,2^*]$
and suppose that there exists $\{\psi_n\}\subset E$
be a sequence such that
\begin{\equ}\label{contradiction-assumption}
c_1\leq\cl_{p_n}(\psi_n)\leq c_2 \quad \text{and} \quad
\cl_{p_n}'(\psi_n)\to0
\end{\equ}
for some constants $c_1,c_2>0$

\begin{Lem}\label{boundedness}
Under \eqref{contradiction-assumption},
$\{\psi_n\}$ is bounded in $E$.
\end{Lem}
\begin{proof}
Since $\cl_{p_n}'(\psi_n)\to0$, we have
\begin{\equ}\label{ee1}
c_2+o(\|\psi_n\|)\geq\cl_{p_n}(\psi_n)-\frac12
\cl_{p_n}'(\psi_n)[\psi_n]=\Big(\frac12-\frac1{p_n}\Big)
\int_{S^m}H(\xi)|\psi_n|^{p_n}d\vol_{\ig_{S^m}}.
\end{\equ}
We also have
\[
o(\|\psi_n\|)=\cl_{p_n}'(\psi_n)[\psi_n^+-\psi_n^-]
=\|\psi_n\|^2-\real\int_{S^m}H(\xi)|\psi_n|^{p_n-2}
(\psi_n,\psi_n^+-\psi_n^-)d\vol_{\ig_{S^m}}.
\]
From this and the H\"older and Sobolev inequalities, we
obtain
\begin{eqnarray}\label{e3}
\|\psi_n\|^2&\leq& \int_{S^m}H(\xi)|\psi_n|^{p_n-1}
|\psi_n^+-\psi_n^-|d\vol_{\ig_{S^m}}
+ o(\|\psi_n\|) \nonumber\\
&\leq&\Big( \int_{S^m}H(\xi)|\psi_n|^{p_n}d\vol_{
\ig_{S^m}} \Big)^{\frac{p_n-1}{p_n}}
\Big( \int_{S^m}H(\xi)|\psi_n^+-\psi_n^-|^{p_n}
d\vol_{\ig_{S^m}} \Big)^{\frac{1}{p_n}}
+ o(\|\psi_n\|) \nonumber\\
&\leq& C|H|_{\infty}^{\frac1{2^*}}\cdot
\big(1+o(\|\psi_n\|)\big)^{\frac{p_n-1}{p_n}}\cdot
|\psi_n^+-\psi_n^-|_{2^*}
+ o(\|\psi_n\|) \nonumber\\
&\leq& C'\big(1+o(\|\psi_n\|)\big)^{
\frac{p_n-1}{p_n}}\cdot\|\psi_n\|+ o(\|\psi_n\|).
\end{eqnarray}
Since $p_n\to p>2$, it follows from \eqref{e3} that
$\{\psi_n\}$ is bounded in $E$. For further reference,
we mention that \eqref{ee1} and \eqref{e3} show that
$\psi_n\to0$ if and only if $\cl_{p_n}(\psi_n)\to0$, and
hence $\|\psi_n\|$ should be bounded away from zero
under the assumption \eqref{contradiction-assumption}.
\end{proof}

\begin{Lem}\label{PS approx}
Let $\{\psi_n\}\subset E$ be a sequence such that
$\cl_{p_n}'(\psi_n)|_{E^-}=o_n(1)$, i.e.,
\begin{\equ}\label{d1}
\sup_{v\in E^-,\, \|v\|=1}\cl_{p_n}'(\psi_n)[v]=o_n(1).
\end{\equ}
Then
\[
\|\psi_n^--h_{p_n}(\psi_n^+)\|\leq
O\big(\big\|\cl_{p_n}'(\psi_n)|_{E^-}\big\|\big).
\]
In particular, if $\{\psi_n\}$ is such that
$\cl_{p_n}'(\psi_n)=o_n(1)$, then
$I_{p_n}'(\psi_n^+)=o_n(1)$.
\end{Lem}
\begin{proof}
For simplicity of notation, let us denote $z_n=\psi_n^+
+h_{p_n}(\psi_n^+)$ and
$v_n=\psi_n^--h_p(\psi_n^+)$. Then we have $v_n\in E^-$
and, by definition of $h_{p_n}$,
\[
0=\cl_{p_n}'(z_n)[v_n]=\real\int_{S^m}(Dz_n,v_n)
d\vol_{\ig_{S^m}}-\real\int_{S^m}H(\xi)|z_n|^{p_n-2}
(z_n,v_n)d\vol_{\ig_{S^m}}.
\]
By \eqref{d1}, it follows that
\[
o(\|v_n\|)=\cl_{p_n}'(\psi_n)[v_n]=\real\int_{S^m}(D\psi_n,v_n)
d\vol_{\ig_{S^m}}-\real\int_{S^m}H(\xi)|\psi_n|^{p_n-2}
(\psi_n,v_n)d\vol_{\ig_{S^m}}.
\]
And hence we have
\begin{\equ}\label{e1}
\aligned
o(\|v_n\|)&=-\int_{S^m}(Dv_n,v_n)d\vol_{\ig_{S^m}}
+\real\int_{S^m}H(\xi)|\psi_n|^{p_n-2}
(\psi_n,v_n)d\vol_{\ig_{S^m}} \\
&\qquad -\real\int_{S^m}H(\xi)|z_n|^{p_n-2}
(z_n,v_n)d\vol_{\ig_{S^m}}.
\endaligned
\end{\equ}
Remark that the functional $\psi\mapsto|\psi|_p^p$
is convex for any $p\in[2,2^*]$, we have
\[
\real\int_{S^m}H(\xi)|\psi_n|^{p_n-2}
(\psi_n,v_n)d\vol_{\ig_{S^m}}
-\real\int_{S^m}H(\xi)|z_n|^{p_n-2}
(z_n,v_n)d\vol_{\ig_{S^m}}\geq0.
\]
Thus, from \eqref{e1} and $v_n\in E^-$, we can infer that
\begin{\equ}\label{e2}
\big|\cl_{p_n}'(\psi_n)[v_n]\big|\geq-\int_{S^m}(Dv_n,v_n)d\vol_{\ig_{S^m}}
=\|v_n\|^2.
\end{\equ}
And therefore we conclude $\|v_n\|\leq
O\big(\big\|\cl_{p_n}'(\psi_n)|_{E^-}\big\|\big).$

If $\{\psi_n\}$ is such that $\cl_{p_n}'(\psi_n)=o_n(1)$,
then estimate \eqref{e2} implies that $\{z_n\}$ is also
satisfying $\cl_{p_n}'(z_n)=o_n(1)$. Thus, we have
$I_{p_n}'(\psi_n^+)=o_n(1)$.
\end{proof}

Now, let us introduce the functional $\ck_p: E^+\to\R$ by
\[
\ck_p(u)=I_p'(u)[u],  \quad u\in E^+.
\]
It is clear that $\ck_p$ is $C^1$ and its derivative is given
by the formula
\[
\ck_p'(u)[w]=I_p'(u)[w]+I_p''(u)[u,w]
\]
for all $u,w\in E^+$. We also have
$\msn_p=\ck_p^{-1}(0)\setminus\{0\}$.

\begin{Lem}\label{Hp}
For any $u\in E^+$ and $p\in(2,2^*]$, we have
\[
\ck_p'(u)[u]\leq 2\ck_p(u)-\frac{p-2}{p-1}
\int_{S^m}H(\xi)|u+h_p(u)|^p d\vol_{\ig_{S^m}}.
\]
\end{Lem}
\begin{proof}
This estimate follows immediately from a similar estimate as \eqref{eee1}.
\end{proof}

We next have:

\begin{Lem}\label{key1}
Let $p_n\to p$ in $(2,2^*]$ as $n\to\infty$,
if $\{u_n\}\subset E^+$ is bounded,
$\displaystyle\liminf_{n\to\infty}I_{p_n}(u_n)>0$ and
$I_{p_n}'(u_n)\to0$ as $n\to\infty$, then there exists
a sequence $\{t_n\}\subset\R$ such that
$t_nu_n\in\msn_{p_n}$ and $|t_n-1|\leq O\big(
\big\| I_{p_n}'(u_n) \big\|\big)$.
\end{Lem}
\begin{proof}
Since $\displaystyle\liminf_{n\to\infty}I_{p_n}(u_n)>0$,
we have
\[
\liminf_{n\to\infty}\int_{S^m}H(\xi)
|u_n+h_{p_n}(u_n)|^{p_n}d\vol_{\ig_{S^m}}\geq c_0
\]
for some constant $c_0>0$. Now, for $n\in\N$, let us
set $g_n:(0,+\infty)\to\R$ by
\[
g_n(t)=\ck_{p_n}(tu_n).
\]
We then have $tg_n'(t)=\ck_{p_n}'(tu_n)[tu_n]$
for all $t>0$ and $n\in\N$. Hence, by Lemma \ref{Hp},
Taylor's formula and the uniform boundedness of
$g_n'(t)$ on bounded intervals, we get
\[
tg_n'(t)\leq 2g_n(1)-\frac{p_n-2}{p_n-1}\int_{S^m}
H(\xi)|u_n+h_{p_n}(u_n)|^{p_n} d\vol_{\ig_{S^m}}
+C |t-1|
\]
for $t$ close to $1$ with $C>0$ independent of $n$.
Notice that $g_n(1)=I_{p_n}'(u_n)[u_n]\to0$, thus
there exists a small constant $\de>0$ such that
\[
g'(t)<-\de \text{ for all } t\in(1-\de,1+\de) \text{ and }
n \text{ large enough}.
\]
Moreover, we have $g_n(1-\de)>0$ and $g_n(1+\de)<0$.
Then, by Inverse Function Theorem,
\[
\tilde u_n:=g_n^{-1}(0)u_n\in\msn_{p_n}
\cap \span\{u_n\}
\]
is well-defined for all $n$ large enough. Furthermore,
since $|g_n'(t)^{-1}|$ is bounded by a constant, say $c_\de>0$, on $(1-\de,1+\de)$, we consequently get
\[
\|u_n-\tilde u_n\|
=|g_n^{-1}(0)-g_n^{-1}(\ck_{p_n}(u_n))|\cdot\|u_n\|
\leq c_\de|\ck_{p_n}(u_n)|\cdot\|u_n\|.
\]
Now the conclusion follows from
$\ck_{p_n}(u_n)\leq O(\|I_{p_n}'(u_n)\|)$.
\end{proof}

\begin{Cor}\label{key2}
Let $\{\psi_n\}$ satisfies \eqref{contradiction-assumption},
then there exists $\{\tilde u_n\}\subset\msn_{p_n}$
such that
$\|\psi_n-\tilde u_n-h_{p_n}(\tilde u_n)\|\leq
O(\|\cl_{p_n}'(\psi_n)\|)$.
In particular
\[
\max_{t>0}I_{p_n}(t\psi_n^+)=I_{p_n}(\tilde u_n)\leq
\cl_{p_n}(\psi_n)+O\big(\|\cl_{p_n}'(\psi_n)\|_{E^*}^2\big).
\]
\end{Cor}
\begin{proof}
According to Lemma \ref{PS approx}, by setting
$z_n=\psi_n^++h_{p_n}(\psi_n^+)$, we have
\[
\|\psi_n-z_n\|\leq O(\|\cl_{p_n}'(\psi_n)\|_{E^*})
\]
and $\{\psi_n^+\}\subset E^+$ is a sequence satisfying
the assumptions of Lemma \ref{key1}. Hence, there exits
$t_n>0$ such that $\tilde u_n:=t_n\psi_n^+\in \msn_{p_n}$
and
\begin{\equ}\label{ee2}
\aligned
\|\psi_n-\tilde u_n-h_{p_n}(\tilde u_n)\|&\leq
\|\psi_n-z_n\|+|t_n-1|\cdot\|\psi_n^+\|
+\|h_{p_n}(\tilde u_n)-h_{p_n}(\psi_n^+)\|\\[0.4em]
&\leq O(\|\cl_{p_n}'(\psi_n)\|_{E^*}) +
O(\|I_{p_n}'(\psi_n^+)\|)
\endaligned
\end{\equ}
where we have used an easily checked inequality
\[
\|h_{p_n}(\tilde u_n)-h_{p_n}(\psi_n^+)\|\leq
\|h_{p_n}'(\tau u_n)\|\cdot \|\tilde u_n-\psi_n^+\|
=O(\|\tilde u_n-\psi_n^+\|).
\]
for some $\tau$ between $t_n$ and $1$.
Remark that
$\|I_{p_n}'(\psi_n^+)\|=\|\cl_{p_n}'(z_n)\|_{E^*}$ and,
by using the $C^2$ smoothness of $\cl_{p_n}$, we have
\[
\|I_{p_n}'(\psi_n^+)\|=
\|\cl_{p_n}'(z_n)\|\leq\|\cl_{p_n}'(\psi_n)\|
+O(\|\psi_n-z_n\|)=O(\|\cl_{p_n}'(\psi_n)\|_{E^*})
\]
This together with \eqref{ee2} implies
\[
\|\psi_n-\tilde u_n-h_{p_n}(\tilde u_n)\|\leq
O(\|\cl_{p_n}'(\psi_n)\|_{E^*}).
\]

Now, by Taylor's formula and the uniform boundedness of $\cl_{p_n}''(\psi_n)$,
we can obtain
\[
\aligned
\cl_{p_n}(\psi_n)&=\cl_{p_n}(\tilde u_n+h_{p_n}
(\tilde u_n))+\cl_{p_n}'(\tilde u_n+h_{p_n}
(\tilde u_n))[\psi_n-\tilde u_n-h_{p_n}
(\tilde u_n)]+O(\|\cl_{p_n}'(\psi_n)\|_{E^*}^2)\\
&=I_{p_n}(\tilde u_n)+I_{p_n}'(\tilde u_n)
[\psi_n^+-\tilde u_n]+O(\|\cl_{p_n}'(\psi_n)\|_{E^*}^2).
\endaligned
\]
Notice that $\tilde u_n=t_n\psi_n^+\in\msn_{p_n}$,
we have $I_{p_n}'(\tilde u_n)[\psi_n^+-\tilde u_n]\equiv0$
and this implies the last estimate of the corollary.
\end{proof}

We are now in a position to complete the proof of
Proposition \ref{continuity prop}.

\begin{proof}[Proof of Proposition \ref{continuity prop}]
We proceed by contradiction. Assume to the contrary that
there exist $\theta_0>0$, $\al_n\to0$ and
$\{\psi_n\}\subset E$ such that
\[
c_1\leq\cl_{p_n}(\psi_n)\leq c_2 \quad \text{and} \quad
\|\cl_{p_n}'(\psi_n)\|_{E^*}\leq \al_n
\]
and
\begin{\equ}\label{contradiction}
\max_{t>0}I_{p_n}(t\psi_n^+)>\cl_{p_n}(\psi_n)+\theta_0.
\end{\equ}
Then it is clear that $\{\psi_n\}$ satisfies
\eqref{contradiction-assumption}.
Therefore, by Corollary \ref{key2}, we should have that
\[
\max_{t>0}I_{p_n}(t\psi_n^+)=I_{p_n}(\tilde u_n)\leq
\cl_{p_n}(\psi_n)+O\big(\|\cl_{p_n}'(\psi_n)\|_{E^*}^2\big).
\]
This contradicts \eqref{contradiction}.
\end{proof}

\subsection{A Rayleigh type quotient}

For any $p\in(2,2^*]$ we define the functional
\[
\rr_{p}: E\setminus\{0\}\to\R, \quad
\psi\mapsto
\frac{\int_{S^m}(D\psi,\psi)d\vol_{\ig_{S^m}}}
{\big(\int_{S^m}H(\xi)|\psi|^pd\vol_{\ig_{S^m}}
\big)^{\frac2p}}.
\]
Here we remark that $\rr_{p}$ is a differentiable
functional and its derivation is given by
\begin{\equ}\label{rr-derivative}
\rr_{p}'(\psi)[\va]=\frac2{A(\psi)^{\frac2p}}\Big[
\real\int_{S^m}(D\psi,\va)d\vol_{\ig_{S^m}}
-\frac{\rr_{p}(\psi)}p A(\psi)^{\frac{2-p}p}\cdot
A'(\psi)[\va] \Big]
\end{\equ}
where (for simplicity) we have used the notations
\[
A(\psi):=\int_{S^m}H(\xi)|\psi|^pd\vol_{\ig_{S^m}}
\quad \text{and} \quad
A'(\psi)[\va]=p\,\real\int_{S^m}H(\xi)|\psi|^{p-2}(\psi,\va)
d\vol_{\ig_{S^m}}.
\]
Let $u\in E^+\setminus\{0\}$, we define
\[
\pi_u: E^-\to\R, \quad v\mapsto \rr_{p}(u+v)
=\frac{\|u\|^2-\|v\|^2}{A(u+v)^{\frac2p}}.
\]
Then we see that $\sup_{v\in E^-}\pi_u(v)>0$ is attained by some $v_u\in E^-$.
Notice that for any positive $c>0$ the set $\{v\in E^-: \pi_u(v)\ge c\}$ is strictly convex because the map
\[
  v\mapsto \|u\|^2 - \|v\|^2  - cA(u+v)^{\frac2p}
\]
is strictly concave on $E^-$. Hence the maximum point $v_u\in E^-$ is uniquely determined.


Now, let's define the map
\[
\msj_{p}: E^+\setminus\{0\}\to E^-, \quad u\mapsto v_u
\text{ which is the maximum point of } \pi_u
\]
and the functional $\msf_{p}: E^+\setminus\{0\}\to\R$
\[
\msf_{p}(u)=\rr_{p}(u+\msj_{p}(u))=\max_{v\in E^-}
\rr_{p}(u+v).
\]
Remark that $\rr_{p}(t\psi)\equiv\rr_{p}(\psi)$
for all $t>0$, thus we have $\msj_{p}(tu)=t\msj_{p}(u)$.
Moreover, since $\msf_p(u)>0$, we must have
$\|\msj_p(u)\|<\|u\|$ for all $u\in E^+\setminus\{0\}$.

\begin{Lem}\label{msf=I}
$\msf_p(u)=\big(\frac{2p}{p-2}I_p(u)\big)^{\frac{p-2}p}$
for $u\in\msn_p$.
\end{Lem}
\begin{proof}
Let $u\in\msn_p$, then
\[
0=I_p'(u)[u]=\int_{S^m}\big( D(u+h_p(u)),u+h_p(u) \big)
d\vol_{\ig_{S^m}}-\int_{S^m}H(\xi)|u+h_p(u)|^p
d\vol_{\ig_{S^m}}.
\]
Hence $I_p(u)=I_p(u)-\frac12I_p'(u)[u]=\frac{p-2}{2p}
\int_{S^m}H(\xi)|u+h_p(u)|^p d\vol_{\ig_{S^m}}$.

On the other hand, by \eqref{rr-derivative}, we have
\[
\rr_p'(u+h_p(u))[v]\equiv0 \quad \forall v\in E^-.
\]
This, together with the fact $\rr_p(u+h_p(u))>0$, suggests
that $\msj_p(u)=h_p(u)$ for $u\in\msn_p$. And therefore
we get
\[
\msf_p(u)=\Big( \int_{S^m}H(\xi)|u+h_p(u)|^p
d\vol_{\ig_{S^m}} \Big)^{1-\frac2p}=
\Big( \frac{2p}{p-2}I_p(u) \Big)^{\frac{p-2}p}
\]
\end{proof}

By Lemma \ref{msf=I} and the fact
\[
\forall u\in E^+\setminus\{0\} \text{ there uniquely exists }
t=t(u)>0 \text{ such that }  t(u)u\in\msn_p,
\]
we have critical points of $\msf_p$ and $I_p$ are in
one-to-one correspondence via the map $u\mapsto t(u)u$
from $E^+\setminus\{0\}$ to $\msn_p$.

Next, for any $p\in(2,2^*]$, we define
\begin{\equ}\label{tau}
\tau_p=\inf_{u\in E^+\setminus\{0\}}\msf_p(u).
\end{\equ}
By Lemma \ref{msf=I}, we have $\tau_p\in(0,+\infty)$.
In order to show properties of the map $p\mapsto \tau_p$,
 we shall first prove the following:

\begin{Lem}\label{msj converge}
Let $\{p_n\}\subset(2,2^*)$ be an increasing sequence that
converges to $p\leq2^*$.
For each $u\in E^+\setminus\{0\}$, we have
$\msj_n(u):=\msj_{p_n}(u)\to\msj_p(u)$ as $n\to\infty$.
\end{Lem}
\begin{proof}
We fix $u\in E^+\setminus\{0\}$ and, by noting that
$\|\msj_n(u)\|<\|u\|$, we can assume without loss of generality
that $\msj_n(u)\rightharpoonup v\in E^-$ as $n\to\infty$.
Moreover, up to a subsequence, we may have
\[
\lim_{n\to\infty}\Big(\int_{S^m}H(\xi)|u+\msj_n(u)|^{p_n}
d\vol_{\ig_{S^m}}\Big)^{\frac1{p_n}}=\ell>0
\]

Remark that for each $\psi\in E$
\[
q\mapsto\Big(\int_{S^m}H(\xi)|\psi|^q d\vol_{\ig_{S^m}}\Big)^{\frac1q}
\quad \text{is nondecreasing}
\]
as we have assumed $\int_{S^m}H(\xi)d\vol_{\ig_{S^m}}=1$.
Then
\begin{\equ}\label{msf-p-decreasing}
\aligned
\msf_{p_{n+1}}(u)&=
\rr_{p_{n+1}}(u+\msj_{n+1}(u))\leq
\rr_{p_{n}}(u+\msj_{n+1}(u))\\
&\leq\rr_{p_{n}}(u+\msj_{n}(u))=\msf_{p_{n}}(u).
\endaligned
\end{\equ}
Hence we have $\{\msf_{p_n}(u)\}$ is a non-increasing
sequence and $\msf_{p_n}(u)\to\tau>0$ as $n\to\infty$.

Choose $t_n>0$ such that $t_nu\in\msn_{p_n}$,
as was argued in Lemma \ref{msf=I},
we shall have $t_n\msj_n(u)=\msj_n(t_nu)
=h_{p_n}(t_nu)$. And hence, we can deduce
\[
\msf_{p_n}(u)=\rr_{p_{n}}(t_nu+t_n\msj_{n}(u))
=\Big( \int_{S^m}H(\xi)|t_nu+t_n\msj_n(u)|^{p_n}
d\vol_{\ig_{S^m}}\Big)^{\frac{p_n-2}{p_n}}.
\]
And therefore we have
\[
\lim_{n\to\infty}t_n=t_0:=\frac{\tau^{\frac1{p-2}}}\ell>0
\]
and $t_n\msj_n(u)\rightharpoonup t_0v$ as $n\to\infty$.

Now, let us take arbitrarily $w\in W(u)=\span\{u\}\op E^-$.
Since $t_nu\in\msn_{p_n}$, we get
\[
\aligned
&\int_{S^m}(D(t_0u+t_0v),w)d\vol_{\ig_{S^m}}
-\real\int_{S^m}H(\xi)|t_0u+t_0v|^{p-2}(t_0u+t_0v,w)
d\vol_{\ig_{S^m}}\\
&\qquad=\lim_{n\to\infty}\bigg[
\int_{S^m}(D(t_nu+t_n\msj_n(u)),w)d\vol_{\ig_{S^m}}\\
&\qquad \qquad \qquad
-\real\int_{S^m}H(\xi)|t_nu+t_n\msj_n(u)|^{p_n-2}
(t_nu+t_n\msj_n(u),w) d\vol_{\ig_{S^m}} \bigg] \\
&\qquad=\lim_{n\to\infty}I_{p_n}(t_nu)[w]=0
\endaligned
\]
This implies, by using the fact $t_0>0$, $t_0v=h_p(t_0u)$ and
$t_0u\in\msn_p$. And thus $v=\msj_p(u)$ and
$\msj_n(u)\rightharpoonup\msj_p(u)$ as $n\to\infty$.

To complete the proof, let us now assume to the contrary that
$\msj_n(u)\nrightarrow\msj_p(u)$ (up to any subsequence).
Then we must have
$\displaystyle\varliminf_{n\to\infty}\|\msj_n(u)\|>\|\msj_p(u)\|$.
And hence we get
\begin{eqnarray*}
\int_{S^m}H(\xi)|t_0u+t_0\msj_p(u)|^p d\vol_{\ig_{S^m}}
&=&\int_{S^m}(D(t_0u+t_0\msj_p(u)),t_0u+t_0\msj_p(u))
d\vol_{\ig_{S^m}}\\
&=&\|t_0u\|^2-\|t_0\msj_p(u)\|^2\\
&>&\varliminf_{n\to\infty}\big(\|t_nu\|^2-\|t_n\msj_n(u)\|^2\big)\\
&=&\varliminf_{n\to\infty}
\int_{S^m}H(\xi)|t_nu+t_n\msj_n(u)|^{p_n}
d\vol_{\ig_{S^m}}\\
&\geq&\int_{S^m}H(\xi)|t_0u+t_0\msj_p(u)|^p d\vol_{\ig_{S^m}}
\end{eqnarray*}
where the last inequality follows from Fatou's lemma. And
this estimate is obviously impossible.
\end{proof}

\begin{Prop}\label{key3}
The function $(2,2^*]\to(0,+\infty)$, $p\mapsto\tau_p$ is
\begin{itemize}
\item[$(1)$] non-increasing,

\item[$(2)$] continuous from the left.
\end{itemize}
\end{Prop}
\begin{proof}
Since $(1)$ is evident as was already
shown in \eqref{msf-p-decreasing},
we only need to prove $(2)$.

Given $p\in(2,2^*]$, we choose $u\in E^+\setminus\{0\}$
such that $\msf_p(u)\leq\tau_p+\epsilon$. Observe that
for all $p'\leq p$
\[
\msf_{p'}(u)=\rr_{p'}(u+\msj_{p'}(u))
=\frac{\Big( \int_{S^m}H(\xi)|u+\msj_{p'}(u)|^p
d\vol_{\ig_{S^m}}
\Big)^{\frac2p}}{\Big( \int_{S^m}H(\xi)|u+\msj_{p'}(u)|^{p'}
d\vol_{\ig_{S^m}}
\Big)^{\frac2{p'}}}\rr_{p}(u+\msj_{p'}(u)).
\]
Thanks to Lemma \ref{msj converge}, the function
\[
p'\mapsto\Big( \int_{S^m}H(\xi)|u+\msj_{p'}(u)|^{p'}
d\vol_{\ig_{S^m}}\Big)^{\frac1{p'}}
\]
is continuous from the left. Hence, if $p'$ is sufficiently close
to $p$, then
\[
\tau_{p'}\leq\msf_{p'}(u)\leq\rr_{p}(u+\msj_{p'}(u))+\epsilon
\leq\tau_p+2\epsilon.
\]
Because $p\mapsto\tau_p$ is non-increasing, the statement
follows.
\end{proof}

\begin{Rem}\label{ground state value}
Recall that the sphere of constant sectional curvature $1$
carries a Killing spinor $\psi^*$ with length $1$
to the constant $-\frac12$, that is, $\psi^*$ satisfies
$|\psi^*|\equiv1$ and
\[
\nabla_X\psi^*=-\frac12 X\cdot\psi^*, \quad
\forall X\in\Ga(TM)
\]
where $\cdot$ denotes the Clifford multiplication.
Therefore we have  $D\psi^*=\frac m2\psi^*$.
And it follows from \cite[Section 4]{Ammann2009} that
\[
\inf_{u\in E^+\setminus\{0\}}\max_{v\in E^-}
\frac{\int_{S^m}(D(u+v),u+v)d\vol_{\ig_{S^m}}}
{\big(\int_{S^m}|u+v|^{2^*}d\vol_{\ig_{S^m}}
\big)^{\frac2{2^*}}}
=\frac{\int_{S^m}(D\psi^*,\psi^*)d\vol_{\ig_{S^m}}}
{\big(\int_{S^m}|\psi^*|^{2^*}d\vol_{\ig_{S^m}}
\big)^{\frac2{2^*}}}
=(\frac m2)\om_m^{\frac1m}.
\]
Then, we can derive that
$\tau_{2^*}\geq\big(\frac1{H_{max}} \big)^{\frac{m-1}m}\big(
\frac m2\big)\om_m^{\frac1m}$.
Thanks to the technical argument in \cite{AGHM}, we know
that $\tau_{2^*}\leq \big(\frac1{H_{max}} \big)^{\frac{m-1}m}\big(
\frac m2\big)\om_m^{\frac1m}$.
Hence we have
\begin{\equ}\label{tau value}
\tau_{2^*}=
 \Big(\frac1{H_{max}} \Big)^{\frac{m-1}m}\big(
\frac m2\big)\om_m^{\frac1m}.
\end{\equ}
\end{Rem}

\section{Blow-up analysis}\label{Blow-up section}
In this section we choose $\{p_n\}$ to be an strictly increasing
sequence such that $\displaystyle\lim_{n\to\infty}p_n=2^*$.
In what follows, we shall investigate the possible convergent
properties of solutions to the equation
\begin{\equ}\label{pn-equ}
D\psi=H(\xi)|\psi|^{p_n-2}\psi \quad \text{on } S^m
\end{\equ}
with some specific energy constraints.  Our arguments will be
divided into three parts. In Subsection \ref{alternative sec}
we establish a kind of alternative behavior for solutions of
\eqref{pn-equ}, which shows either compactness or
blow-up phenomenon. In Subsection \ref{blow-up sec},
we describe the specific blow-up phenomenon which appears
if we exclude the compactness. And in
Subsection \ref{stereo sec} we deal with the stereographic
projected view of the blow-up behavior.

\subsection{An alternative property}\label{alternative sec}
As before, we will denote $\cl_{p_n}$ the energy functional
associated to Equation \eqref{pn-equ} and
$(h_{p_n},I_{p_n})$
the reduction couple for $\cl_{p_n}$. Our alternative
result comes as follows:

\begin{Prop}\label{alternative prop}
Suppose $\{\psi_n\}\subset E$ is a sequence such that
\begin{\equ}\label{key-assumption}
\frac1{2m}(\tau_{2^*})^m\leq\cl_{p_n}(\psi_n)
\leq\frac1{m}(\tau_{2^*})^m-\theta \quad \text{and} \quad
\cl_{p_n}'(\psi_n)\to0
\end{\equ}
for some constant $\theta>0$. Then, up to a subsequence,
either $\psi_n\rightharpoonup 0$ or
$\psi_n\to\psi_0\neq0$
in $E$.
\end{Prop}
\begin{proof}
Notice that $\{\psi_n\}$ is bounded in $E$ (by
Lemma \ref{boundedness}), and we may then
assume that $\psi_n\rightharpoonup \psi_0$ in $E$ as
$n\to\infty$ with some $\psi_0$ satisfying the equation
\begin{\equ}\label{eq0}
D\psi_0=H(\xi)|\psi_0|^{2^*-2}\psi_0 \quad \text{on } S^m.
\end{\equ}

Set $\bar\psi_n=\psi_n-\psi_0$. Then we have $\bar\psi_n$
satisfy
\[
\aligned
D\bar\psi_n&=H(\xi)|\psi_n|^{p_n-2}\psi_n
-H(\xi)|\bar\psi_n|^{p_n-2}\bar\psi_n
-H(\xi)|\psi_0|^{p_n-2}\psi_0 \\[0.4em]
&\qquad +H(\xi)|\psi_0|^{p_n-2}\psi_0
-H(\xi)|\psi_0|^{2^*-2}\psi_0\\[0.4em]
&\qquad +H(\xi)|\bar\psi_n|^{p_n-2}\bar\psi_n+o_n(1)
\endaligned
\]
where $o_n(1)\to0$ as $n\to\infty$ in $E^*$.

To proceed, we set
\[
\Phi_n=H(\xi)|\psi_n|^{p_n-2}\psi_n
-H(\xi)|\bar\psi_n|^{p_n-2}\bar\psi_n
-H(\xi)|\psi_0|^{p_n-2}\psi_0
\]
It is easy to see that there exists $C>0$ (independent of $n$)
such that
\begin{\equ}\label{e4}
|\Phi_n|\leq C|\bar\psi_n|^{p_n-2}|\psi_0|
+C|\psi_0|^{p_n-2}|\bar\psi_n|.
\end{\equ}
Thanks to Egorov theorem, for any $\epsilon>0$, there exists
$\Om_\epsilon\subset S^m$ such that
$\meas\{S^m\setminus \Om_\epsilon\}<\epsilon$ and
$\bar\psi_n\to0$ uniformly on $\Om_\epsilon$ as
$n\to\infty$. Therefore, by \eqref{e4} and the H\"older
inequality, we have
\begin{eqnarray}\label{integral1}
\real\int_{S^m}(\Phi_n,\va)d\vol_{\ig_{S^m}}&=&
\real\int_{S^m\setminus \Om_\epsilon}
(\Phi_n,\va)d\vol_{\ig_{S^m}}
+\real\int_{\Om_\epsilon}(\Phi_n,\va)d\vol_{\ig_{S^m}}
\nonumber\\
&\leq&C\Big(\int_{S^m\setminus \Om_\epsilon}
|\bar\psi_n|^{2^*}d\vol_{\ig_{S^m}}
\Big)^{\frac{p_n-2}{2^*}}
\Big(\int_{S^m\setminus \Om_\epsilon}
|\psi_0|^{2^*}d\vol_{\ig_{S^m}}
\Big)^{\frac1{2^*}} \|\va\|  \nonumber\\
& & +\,C\Big(\int_{S^m\setminus \Om_\epsilon}
|\psi_0|^{2^*}d\vol_{\ig_{S^m}}
\Big)^{\frac{p_n-2}{2^*}}
\Big(\int_{S^m\setminus \Om_\epsilon}
|\bar\psi_n|^{2^*}d\vol_{\ig_{S^m}}
\Big)^{\frac1{2^*}} \|\va\|  \nonumber\\
& &+\int_{\Om_\epsilon}|\Phi_n|\cdot|\va|
d\vol_{\ig_{S^m}}.
\end{eqnarray}
for arbitrary $\va\in E$ with $\|\va\|\leq1$. It is
evident that
the last integral in \eqref{integral1} converges to $0$ as
$n\to\infty$ and the remaining integrals tends to $0$
uniformly in $n$ as $\epsilon\to0$. Thus, we get
$\Phi_n=o_n(1)$ in $E^*$. Noting that
$q\mapsto H(\cdot)|\psi_0|^{q-2}\psi_0$ is continuous
in $E^*$, hence we have
\begin{\equ}\label{eq1}
D\bar\psi_n=H(\xi)|\bar\psi_n|^{p_n-2}\bar\psi_n+o_n(1)
\quad \text{in } E^*.
\end{\equ}

Now assume $\psi_0\neq0$. If there exists a subsequence
such that $\cl_{p_n}(\bar\psi_n)\to0$, then it follows from
the proof of Lemma \ref{boundedness} that
we must have $\bar\psi_n\to0$. So we now assume that,
up to any subsequence, $\cl_{p_n}(\bar\psi_n)\not\to0$.

Since $\psi_0$ is a non-trivial solution to \eqref{eq0}, by
Lemma \ref{msf=I} and the definition of $\tau_{2^*}$
(c.f. \eqref{tau}), we have
\[
\tau_{2^*}\Big( \int_{S^m}H(\xi)|\psi_0|^{2^*}
d\vol_{\ig_{S^m}} \Big)^{\frac2{2^*}}\leq
\int_{S^m}(D\psi_0,\psi_0)d\vol_{\ig_{S^m}}
=\int_{S^m}H(\xi)|\psi_0|^{2^*}
d\vol_{\ig_{S^m}}
\]
and thus
\begin{\equ}\label{e5}
\int_{S^m}(D\psi_0,\psi_0)d\vol_{\ig_{S^m}}
\geq(\tau_{2^*})^{\frac{2^*}{2^*-2}}
\end{\equ}
On the other hand, by \eqref{eq1} and
$\cl_{p_n}(\bar\psi_n)\not\to0$, we have
\[
c_1\leq\cl_{p_n}(\bar\psi_n)\leq c_2 \quad \text{and} \quad
\cl_{p_n}'(\bar\psi_n)\to0
\]
for some $c_1,c_2>0$. Therefore, by Corollary \ref{key2},
Lemma \ref{msf=I} and the uniform boundedness of
the second derivatives of $\cl_{p_n}$ near $\bar\psi_n$,
we can conclude
\[
\tau_{p_n}\leq\msf_{p_n}(\bar\psi_n^+)=\max_{t>0}
\Big( \frac{2p_n}{p_n-2} I_{p_n}(t\bar\psi_n^+)
\Big)^{\frac{p_n-2}{p_n}}
\leq\Big( \frac{2p_n}{p_n-2} \cl_{p_n}(\bar\psi_n)+o_n(1)
\Big)^{\frac{p_n-2}{p_n}}.
\]
This, together with Proposition \ref{key3}, implies
\begin{\equ}\label{e6}
\int_{S^m}(D\bar\psi_n,\bar\psi_n)d\vol_{\ig_{S^m}}
\geq\big(\tau_{p_n}\big)^{\frac{p_n}{p_n-2}}+o_n(1)
=(\tau_{2^*})^{\frac{2^*}{2^*-2}}+o_n(1).
\end{\equ}
And we thus have
\[
\aligned
\cl_{p_n}(\psi_n)&=\frac{p_n-2}{2p_n}
\int_{S^m}(D\psi_n,\psi_n)d\vol_{\ig_{S^m}} + o_n(1)\\
&=\frac{p_n-2}{2p_n}
\int_{S^m}(D\bar\psi_n,\bar\psi_n)d\vol_{\ig_{S^m}}
+\frac{p_n-2}{2p_n}
\int_{S^m}(D\psi_0,\psi_0)d\vol_{\ig_{S^m}}
+ o_n(1) \\
&\geq \frac1m(\tau_{2^*})^{m}+o_n(1)
\endaligned
\]
where the last inequality follows from \eqref{e5} and \eqref{e6}.
This contradicts \eqref{key-assumption}.
\end{proof}

\subsection{Blow-up phenomenon}\label{blow-up sec}

Let $\{\psi_n\}\subset E$ fulfill the assumption
of Proposition \ref{alternative prop}, that is
\eqref{key-assumption}. If $\{\psi_n\}$ has a subsequence
which is compact in $E$, then the same subsequence converges
and the limit spinor $\psi_0$ is a non-trivial solution
to \eqref{eq0}.
Thus we are interested in the case where any subsequence
of $\{\psi_n\}$ does not converge. From now on,
by Proposition \ref{alternative prop}, we may assume
$\psi_n\rightharpoonup0$ in $E$ as $n\to\infty$.

To begin with, we shall first introduce an useful concept
of blow-up set of $\{\psi_n\}$:
\[
\Ga:=\Big\{
a\in M:\, \varliminf_{r\to0}\varliminf_{n\to\infty}
\int_{B_r(a)}|\psi_n|^{p_n}d\vol_{\ig_{S^m}}\geq \de_0
\Big\}
\]
where $B_r(a)\subset S^m$ is the distance ball of radius $r$
with respect to the metric $\ig_{S^m}$ and $\de_0>0$
is a positive constant. The value of $\de_0$ can be determined
in the sense of the following lemma:

\begin{Lem}\label{blow-up set def}
Let $\{\psi_n\}$ be as above. Then there exists $\de_0>0$
such that $\Ga\neq\emptyset$.
\end{Lem}
\begin{proof}
Assume to the contrary that  $\Ga=\emptyset$ for any choice
of $\de_0$ (up to any subsequence of $\{\psi_n\}$).
Then we may fix $\de_0$ arbitrary small.
And we have, for any $a\in M$, there exists $r_0>0$ such that
\begin{\equ}\label{e7}
\int_{B_{2r_0}(a)}|\psi_n|^{p_n}
d\vol_{\ig_{S^m}}<\de_0.
\end{\equ}
for all $n$ large.

Let us take $\eta\in C^\infty(S^m)$ such that $\eta\equiv1$
on $B_{r_0}(a)$ and $\eta\equiv 0$ on
$S^m\setminus B_{2r_0}(a)$. Since
$\cl_{p_n}'(\psi_n)=o_n(1)$ in $E^*$, we can obtain
\begin{eqnarray*}
D(\eta \psi_n)&=& \eta D\psi_n+\nabla\eta\cdot\psi_n\\
&=&\eta(\xi) H(\xi)|\psi_n|^{p_n-2}\psi_n
+\nabla\eta\cdot\psi_n + o_n(1)
\end{eqnarray*}
where $\cdot$ denotes the Clifford multiplication and
$o_n(1)\to0$ in $E^*$ as $n\to\infty$.

Noting that there exists $C>0$ such that
\[
\|\psi\|\leq C\|D\psi\|_{E^*}+ C|\psi|_2 \quad
\forall \psi\in E.
\]
Thus, we have
\begin{eqnarray}\label{e8}
\|\eta\psi_n\|&\leq& C\|D(\eta\psi_n)\|_{E^*}
+C|\eta\psi_n|_2 \nonumber\\[0.4em]
&\leq& C\big\|\eta(\xi) H(\xi)|\psi_n|^{p_n-2}\psi_n
+\nabla\eta\cdot\psi_n \big\|_{E^*}+C|\eta\psi_n|_2
+o_n(1) \nonumber\\[0.4em]
&\leq& C\big\|\eta(\xi)
H(\xi)|\psi_n|^{p_n-2}\psi_n\big\|_{E^*}
+C\big\|\nabla\eta\cdot\psi_n \big\|_{E^*}
+C|\eta\psi_n|_2 +o_n(1) .
\end{eqnarray}
Remark that, by the Sobolev embedding
$L^2\hookrightarrow E^*$, we have
$\big\|\nabla\eta\cdot\psi_n \big\|_{E^*}\leq
C'|\nabla\eta\cdot\psi_n|_2$ for some constant $C'>0$.
Moreover, by the Sobolev embedding
$E\hookrightarrow L^{p_n}$ and the
H\"older inequality, there holds
\begin{eqnarray}\label{e9}
\real\int_{S^m}\eta(\xi)H(\xi)|\psi_n|^{p_n-2}(\psi_n,\va)
d\vol_{\ig_{S^m}}&\leq& |H|_\infty|\va|_{p_n}
|\eta\psi_n|_{p_n}\Big( \int_{B_{2r_0}(a)}|\psi_n|^{p_n}
d\vol_{\ig_{S^m}}\Big)^{\frac{p_n-2}{p_n}} \nonumber\\
&\leq& C'' |H|_\infty \de_0^{\frac{p_n-2}{p_n}}
\|\eta\psi_n\|\cdot\|\va\|
\end{eqnarray}
for all $\va\in E$,
where $C''>0$ depends only on $S^m$ and, in the last
inequality, we have used \eqref{e7}.

Recall that we have $\{p_n\}$ is a strictly increasing sequence
such that $p_n\to 2^*$ as $n\to\infty$.
Therefore, we may choose $\de_0$ so small that
$CC''|H|_\infty \de_0^{\frac{p_n-2}{p_n}}<\frac12$. Then
by \eqref{e8} and \eqref{e9}, we get
\[
\|\eta\psi_n\|\leq CC'|\nabla\eta\cdot\psi_n|_2
+C|\eta\psi_n|_2+o_n(1).
\]
Let us mention that we have assumed
$\psi_n\rightharpoonup0$. Hence, by the compact embedding
$E\hookrightarrow L^2$, we are arrived at
$\|\eta\psi_n\|=o_n(1)$ as $n\to\infty$.

Since $a\in  S^m$ is arbitrary and $S^m$ is compact, we can
conclude $\psi_n\to0$ in $E$ which contradicts
\eqref{key-assumption}.
\end{proof}

Another useful concept in this context is the concept of the
concentration function introduced in \cite{Lions1, Lions2, BC}.
For $r\geq0$, let us define
\[
\Theta_n(r)=\sup_{a\in S^m}\int_{B_r(a)}|\psi_n|^{p_n}
d\vol_{\ig_{S^m}}.
\]
Choose $\bar\de>0$ small, say $\bar\de<\de_0$ where
$\de_0$ is as in Lemma \ref{blow-up set def}. Then there
exist a decreasing sequence $\{R_n\}\subset\R$, $R_n\to0$
as $n\to\infty$ and $\{a_n\}\subset S^m$ such that
\begin{\equ}\label{choice of an}
\Theta_n(R_n)=\int_{B_{R_n}(a_n)}|\psi_n|^{p_n}
d\vol_{\ig_{S^m}}=\bar\de.
\end{\equ}
Up to a subsequence if necessary, we assume that
$a_n\to a\in S^m$ as $n\to\infty$.

Now, let us define the rescaled geodesic normal coordinates
near each $a_n$ via the formula
\[
\mu_n(x)=\exp_{a_n}(R_n x).
\]
Denoting $B_R^0=\big\{ x\in\R^m:\, |x|<R\big\}$, where
$|\cdot|$ is the Euclidean norm in $\R^m$, we have a
conformal equivalence $(B_R^0, \, R_n^{-2}\mu_n^*\ig_{S^m})
\cong (B_{R_nR}(a_n),\, \ig_{S^m})\subset S^m$ for all
large $n$.

For ease of notation, we set
$\ig_n=R_n^{-2}\mu_n^*\ig_{S^m}$. Writing the metric
$\ig_{S^m}$ in geodesic normal coordinates centered in $a$,
one immediately sees that, for any $R>0$, $\ig_n$ converges
to the Euclidean metric in $C^\infty(B_R^0)$ as $n\to\infty$.

Now, following Proposition \ref{conformal formula} and
the idea of local trivialization introduced
in Subsection \ref{BG sec},
we can conclude that the coordinate map $\mu_n$
induces a spinor identification
 $\ov{(\mu_n)}_*:\mbs_x(B_R^0,\ig_n)\to
\mbs_{\mu_n(x)}(B_{R_nR}(a_n),\ig_{S^m})$.
If we define spinors $\phi_n$ on $B_R^0$ by
\begin{\equ}\label{phi-n}
\phi_n=R_n^{\frac{m-1}2}\ov{(\mu_n)}_*^{\,-1}\circ\psi_n
\circ\mu_n,
\end{\equ}
then a straightforward calculation shows that
\begin{\equ}\label{c1}
D_{\ig_n}\phi_n=R_n^{\frac{m+1}2}\ov{(\mu_n)}_*^{\,-1}
\circ(D\psi_n)\circ\mu_n,
\end{\equ}
\begin{\equ}\label{c2}
\int_{B_R^0}(D_{\ig_n}\phi_n,\phi_n)d\vol_{\ig_n}
=\int_{B_{R_nR}(a_n)}(D\psi_n,\psi_n)d\vol_{\ig_{S^m}},
\end{\equ}
\begin{\equ}\label{c3}
\int_{B_R^0}|\phi_n|^{2^*}d\vol_{\ig_n}=
\int_{B_{R_nR}(a_n)}|\psi_n|^{2^*}d\vol_{\ig_{S^m}},
\end{\equ}
\begin{\equ}\label{c4}
\int_{B_R^0}|\phi_n|^{p_n}d\vol_{\ig_n}
=R_n^{-\frac{m-1}2(2^*-p_n)}
\int_{B_{R_nR}(a_n)}|\psi_n|^{p_n}d\vol_{\ig_{S^m}}.
\end{\equ}
Moreover, since $\{\psi_n\}$ is bounded in $E$, we have
\begin{\equ}\label{c5}
\sup_{n\geq1}\int_{B_R^0}|\phi_n|^{2^*}d\vol_{\ig_n}\leq
\sup_{n\geq1}\int_{S^m}|\psi_n|^{2^*}d\vol_{\ig_{S^m}}
<+\infty
\end{\equ}
for any $R>0$.

\begin{Lem}\label{concentration parameter}
There is $\bar\lm>0$ such that
\[
\bar\lm\leq\varliminf_{n\to\infty}R_n^{\frac{m-1}2(2^*-p_n)}
\leq\varlimsup_{n\to\infty}R_n^{\frac{m-1}2(2^*-p_n)}\leq1.
\]
\end{Lem}
\begin{proof}
Since we have
\[
\int_{B_{R_n}(a_n)}|\psi_n|^{p_n}
d\vol_{\ig_{S^m}}=\bar\de,
\]
it follows from \eqref{c4} and H\"older inequality that
\[
\bar\de=\int_{B_{R_n}(a_n)}|\psi_n|^{p_n}
d\vol_{\ig_{S^m}}
\leq\Big( \int_{B_1^0}|\phi_n|^{2^*}d\vol_{\ig_n}
\Big)^{\frac{p_n}{2^*}}\Big( \int_{B_1^0}d\vol_{\ig_n}
\Big)^{\frac{2^*-p_n}{2^*}}
R_n^{\frac{m-1}2(2^*-p_n)}.
\]
Noting that $\ig_n$ converges to the Euclidean metric
in the $C^\infty$-topology on $B_1^0$, we can conclude
immediately from $p_n\to2^*$ and \eqref{c5} that
\[
\bar\de\leq C \cdot R_n^{\frac{m-1}2(2^*-p_n)}
\]
for some constant $C>0$.

On the other hand, suppose there exists some $\de>0$
such that $R_n^{\frac{m-1}2(2^*-p_n)}\geq 1+\de$
for all large $n$. Then, we must have
\[
\ln R_n\geq \frac{2\ln(1+\de)}{(m-1)(2^*-p_n)}\to+\infty
\quad \text{as } n\to\infty.
\]
This implies $R_n\to+\infty$ which is absurd.
\end{proof}

Moreover, we have

\begin{Lem}\label{converge in H-loc}
Let $\{\phi_n\}$ be defined in \eqref{phi-n}. Define
\[
\bar L_n=D_{\ig_n}\phi_n-R_n^{\frac{m-1}2(2^*-p_n)}
H\circ\mu_n(\cdot)|\phi_n|^{p_n-2}\phi_n
\in H^{-\frac12}_{loc}(\R^m,\mbs_m).
\]
Then $\bar L_n\to0$ in $H^{-\frac12}_{loc}(\R^m,\mbs_m)$
in the sense that, for any $R>0$, there holds
\[
\sup\big\{ \inp{\bar L_n}{\va}:\, \va\in H^{\frac12}
(\R^m,\mbs_m) ,\, \supp\va\subset B_R^0,
\|\va\|_{H^{\frac12}}\leq1\big\}\to0
\]
as $n\to\infty$.
\end{Lem}
\begin{proof}
According to \eqref{phi-n} and \eqref{c1}, we get
\[
\aligned
D_{\ig_n}\phi_n-R_n^{\frac{m-1}2(2^*-p_n)}
H\circ\mu_n(\cdot)|\phi_n|^{p_n-2}\phi_n
&=R_n^{\frac{m+1}2}\ov{(\mu_n)}_*^{\,-1}\circ\big(
D\psi_n-H(\cdot)|\psi_n|^{p_n-2}\psi_n \big)\circ\mu_n
\\[0.4em]
&=R_n^{\frac{m+1}2}\ov{(\mu_n)}_*^{\,-1}\circ L_n\circ
\mu_n
\endaligned
\]
where $L_n=D\psi_n-H(\cdot)|\psi_n|^{p_n-2}\psi_n
\in E^*$.

Let $\va\in H^{\frac12}(\R^m,\mbs_m)$ be such that
$ \supp\va\subset B_R^0$ and $\|\va\|_{H^{\frac12}}\leq1$.
Then, for all large $n$, we get $d\vol_{\ig_n}
=(1+o_n(1))d\vol_{\ig_{\R^m}}$ and
\begin{eqnarray}\label{ee3}
(1+o_n(1))\inp{\bar L_n}{\va}&=&\real\int_{B_{1/R_n}^0}
(\bar L_n,\va) d\vol_{\ig_n}  \nonumber\\
&=&\real\int_{B_{1/R_n}^0}
\big(\ov{(\mu_n)}_*^{\,-1}\circ L_n\circ \mu_n,
R_n^{\frac{m+1}2}\va \big) d\vol_{\ig_n}  \nonumber\\
&=&\real\int_{B_{1/R_n}^0}
\big(\ov{(\mu_n)}_*^{\,-1}\circ L_n\circ \mu_n,
R_n^{-\frac{m-1}2}\va \big) d\vol_{\mu^*\ig_{S^m}}
\nonumber\\
&=&\real\int_{B_1(a_n)}
\big( L_n,
R_n^{-\frac{m-1}2}\ov{(\mu_n)}_*\circ \va
\circ \mu_n^{-1} \big) d\vol_{\ig_{S^m}}.
\end{eqnarray}

Noting that $ \supp\va\subset B_R^0$ and
$\|\va\|_{H^{\frac12}}\leq1$, we can find a constant $C>0$
independent of $n$ and $\va$ such that
$\big\|R_n^{-\frac{m-1}2}\ov{(\mu_n)}_*\circ \va
\circ \mu_n^{-1}\big\|\leq C$. Thus,
by \eqref{key-assumption} and \eqref{ee3},
we can obtain the desired assertion.
\end{proof}

In what follows, by Lemma \ref{concentration parameter},
we may assume that, after taking a subsequence if necessary,
\[
R_n^{\frac{m-1}2(2^*-p_n)}\to
\lm\in[\bar\lm,1].
\]
Since $\{\phi_n\}$ is bounded in $H_{loc}^{\frac12}
(\R^m,\mbs_m)$, we can assume (up to a subsequence)
$\phi_n\rightharpoonup\phi_0$ in $H_{loc}^{\frac12}
(\R^m,\mbs_m)$. Thanks to the compact embedding
$H_{loc}^{\frac12}(\R^m,\mbs_m)
\hookrightarrow L^q(\R^m,\mbs_m)$ for $1\leq q<2^*$,
it is easy to see that $\phi_0\in L^{2^*}(\R^m,\mbs_m)$
satisfies
\[
D_{\ig_{\R^m}}\phi_0=\lm H(a)|\phi_0|^{2^*-2}\phi_0
\quad \text{on } \R^m.
\]
Furthermore, we have

\begin{Lem}\label{phi-n to phi0}
$\phi_n\to\phi_0$ in $H_{\loc}^{\frac12}(\R^m,\mbs_m)$
as $n\to\infty$.
\end{Lem}
\begin{proof}
For ease of notation, we shall set $z_n=\phi_n-\phi_0$.
Let us fix $y\in\R^m$ arbitrarily, then it follows from
\eqref{choice of an} and \eqref{c4} that
\begin{\equ}\label{cc1}
R_n^{\frac{m-1}2(2^*-p_n)}
\int_{B_1^0(y)}|\phi_n|^{p_n} d\vol_{\ig_n}
\leq\bar\de \quad
\text{for all } n \text{ large},
\end{\equ}
where $B_R^0(y)=\big\{x\in\R^m:\, |x-y|<R\big\}$
is the Euclidean ball centered at $y$ for any $R>0$.
And we can also conclude from Fatou lemma that
\begin{\equ}\label{ccc1}
\lm\int_{B_1^0(y)}|\phi_0|^{2^*}d\vol_{\ig_{\R^m}}
\leq\bar\de.
\end{\equ}

Taking a smooth function
$\bt:\R^m\to[0,1]$ such that
$\supp\bt\subset B_1^0(y)$. Since, for any
$\phi\in H^{\frac12}(\R^m,\mbs_m)$, we have the estimate
\[
\|\phi\|_{H^{1/2}}\leq C\|D_{\ig_{\R^m}}\phi
\|_{H^{-1/2}}+ C|\phi|_2
\]
for some constant $C>0$ depends only on $m$,
we soon get the estimate for $\bt^2z_n$ as
\begin{eqnarray}\label{bt zn}
\|\bt^2z_n\|_{H^{1/2}}&\leq&
C\big\|D_{\ig_{\R^m}}(\bt^2z_n)
\big\|_{H^{-\frac12}}+ C|\bt^2z_n|_2 \nonumber \\[0.4em]
&\leq& C\big\|D_{\ig_n}(\bt^2\phi_n)
-D_{\ig_{\R^m}}(\bt^2\phi_0)\big\|_{H^{-1/2}}
+C\big\|(D_{\ig_{\R^m}}-D_{\ig_n})(\bt^2\phi_n)
\big\|_{H^{-1/2}}  \nonumber\\[0.4em]
& & + C|\bt^2z_n|_2.
\end{eqnarray}
Noting that $z_n\to0$ in $L_{loc}^2(\R^m,\mbs_m)$,
we immediately have $|\bt^2z_n|_2=o_n(1)$
as $n\to\infty$. To estimate the second term, we employ
an argument of Isobe \cite[Lemma 5.5]{Isobe-JFA}:
we first observe that
\[
\inp{(D_{\ig_{\R^m}}-D_{\ig_n})(\bt^2\phi_n)}{\va}
=\inp{\bt\phi_n}{\bt(D_{\ig_{\R^m}}-D_{\ig_n}^*)\va}
\]
for any $\va\in H^{\frac12}(\R^m,\mbs_m)$, where
$D_{\ig_n}^*$ is the adjoint of $D_{\ig_n}$ with respect
to the metric $\ig_{\R^m}$.

By recalling that
$\ig_n$ converges to $\ig_{\R^m}$ in $C^\infty$-topology
on bounded domains in $\R^m$, we get
$\bt(D_{\ig_{\R^m}}-D_{\ig_n}^*): H^1(\R^m,\mbs_m)
\to L^2(\R^m,\mbs_m)$ satisfies
\begin{\equ}\label{inter1}
\big\|\bt(D_{\ig_{\R^m}}-D_{\ig_n}^*)  \big\|_{H^1\to L^2}
\to0
\end{\equ}
and $(D_{\ig_{\R^m}}-D_{\ig_n})\bt:H^1(\R^m,\mbs_m)
\to L^2(\R^m,\mbs_m)$ satisfies
\begin{\equ}\label{inter2}
\big\| (D_{\ig_{\R^m}}-D_{\ig_n})\bt \big\|_{H^1\to L^2}
\to0
\end{\equ}
as $n\to\infty$. Then, by taking the dual
of \eqref{inter2}, we get
$\bt(D_{\ig_{\R^m}}-D_{\ig_n}^*): L^2(\R^m,\mbs_m)
\to H^{-1}(\R^m,\mbs_m)$ satisfies
\begin{\equ}\label{inter3}
\big\|\bt(D_{\ig_{\R^m}}-D_{\ig_n}^*)
 \big\|_{L^2\to H^{-1}} \to0
\end{\equ}
as $n\to\infty$.

Therefore, interpolating \eqref{inter1} and \eqref{inter3},
we see that
\[
\big\|\bt(D_{\ig_{\R^m}}-D_{\ig_n}^*)
 \big\|_{H^{1/2}\to H^{-1/2}} \to0
\]
and
\[
\big\|(D_{\ig_{\R^m}}-D_{\ig_n})(\bt^2\phi_n)
\big\|_{H^{-1/2}}\leq\|\bt\phi_n\|_{H^{1/2}}
\big\|\bt(D_{\ig_{\R^m}}-D_{\ig_n}^*)
 \big\|_{H^{1/2}\to H^{-1/2}}\to0
\]
as $n\to\infty$.

To complete the proof, it remains to
estimate the first term in \eqref{bt zn}. Recall that,
by Lemma \ref{converge in H-loc}, we have
\[
D_{\ig_n}\phi_n=R_n^{\frac{m-1}2(2^*-p_n)}
H\circ\mu_n(\cdot)|\phi_n|^{p_n-2}\phi_n+\bar L_n
\]
and $\bar L_n\to0$ in $H^{\frac12}_{loc}(\R^m,\mbs_m)$
as $n\to\infty$. Hence, we deduce
\begin{\equ}\label{ee4}
\aligned
& \|D_{\ig_n}(\bt^2\phi_n)
-D_{\ig_{\R^m}}(\bt^2\phi_0)\big\|_{H^{-1/2}} \\[0.4em]
&\quad\leq
\big\| \bt^2R_n^{\frac{m-1}2(2^*-p_n)}H\circ\mu_n
|\phi_n|^{p_n-2}\phi_n-\bt^2\lm H(a)|\phi_0|^{2^*-2}
\phi_0 \big\|_{H^{-1/2}}\\[0.4em]
&\qquad + \big\| \nabla(\bt^2)\cdot_{\ig_n}\phi_n
-\nabla(\bt^2)\cdot_{\ig_{\R^m}}\phi_0\big\|_{H^{-1/2}}
+o_n(1),
\endaligned
\end{\equ}
where $\cdot_{\ig_n}$ and $\cdot_{\ig_{\R^m}}$
are Clifford multiplication with respect to the metrics
$\ig_n$ and $\ig_{\R^m}$, respectively.

Remark that $L^{\frac{2m}{m+1}}(\R^m,\mbs_m)
\hookrightarrow H^{-\frac12}(\R^m,\mbs_m)$, we have
\[
\big\| \nabla(\bt^2)\cdot_{\ig_n}\phi_n
-\nabla(\bt^2)\cdot_{\ig_{\R^m}}\phi_0\big\|_{H^{-1/2}}
\leq\big| \nabla(\bt^2)\cdot_{\ig_n}\phi_n
-\nabla(\bt^2)\cdot_{\ig_{\R^m}}\phi_0
\big|_{\frac{2m}{m+1}} \to 0
\]
as $n\to\infty$.

On the other hand, since we are working on the bounded
domain $B_1^0(y)\subset\R^m$,we can argue as
\eqref{integral1} to obtain
\[
\aligned
&\big\| \bt^2R_n^{\frac{m-1}2(2^*-p_n)}H\circ\mu_n
|\phi_n|^{p_n-2}\phi_n-\bt^2\lm H(a)|\phi_0|^{2^*-2}
\phi_0 \big\|_{H^{-1/2}} \\[0.3em]
& \quad =\big\|
\bt^2R_n^{\frac{m-1}2(2^*-p_n)}H\circ\mu_n
|z_n|^{p_n-2}z_n \big\|_{H^{-1/2}} + o_n(1)
\endaligned
\]
as $n\to\infty$. And thus, by Sobolev embedding and H\"older
inequality, we have
\begin{\equ}\label{cc2}
\aligned
&\big\| \bt^2R_n^{\frac{m-1}2(2^*-p_n)}H\circ\mu_n
|\phi_n|^{p_n-2}\phi_n-\bt^2\lm H(a)|\phi_0|^{2^*-2}
\phi_0 \big\|_{H^{-1/2}} \\[0.4em]
& \quad \leq C R_n^{\frac{m-1}2(2^*-p_n)}
 \Big(
 \int_{B_1^0(y)}|z_n|^{p_n} d\vol_{\ig_n}
 \Big)^{\frac{p_n-2}{p_n}}\|\bt^2z_n\|_{H^{1/2}}+o_n(1)
\endaligned
\end{\equ}
for some constant $C>0$. Moreover, by \eqref{cc1}
and \eqref{ccc1}, we can infer that
\begin{\equ}\label{cc3}
R_n^{\frac{m-1}2(2^*-p_n)\cdot\frac{1}{p_n}}\Big(
\int_{B_1^0(y)}|z_n|^{p_n} d\vol_{\ig_n}
\Big)^{\frac{1}{p_n}} \leq 2\bar\de^{\frac1{p_n}}+o_n(1).
\end{\equ}
Therefore, combining \eqref{bt zn}, \eqref{ee4}, \eqref{cc2}
and \eqref{cc3}, we can get
\[
\aligned
\|\bt^2z_n\|_{H^{1/2}} &\leq C\big\|
\bt^2R_n^{\frac{m-1}2(2^*-p_n)}H\circ\mu_n
|z_n|^{p_n-2}z_n \big\|_{H^{-1/2}} + o_n(1) \\[0.4em]
&\leq CR_n^{\frac{m-1}2(2^*-p_n)\cdot{\frac2{p_n}}}
\bar\de^{\frac{p_n-2}{p_n}}\|\bt^2z_n\|_{H^{1/2}}+o_n(1)
\\[0.4em]
&\leq C
\bar\de^{\frac{p_n-2}{p_n}}\|\bt^2z_n\|_{H^{1/2}}+o_n(1)
\endaligned
\]
as $n\to\infty$, where we have used
$R_n^{\frac{m-1}2(2^*-p_n)}\to\lm\leq1$ in the last
inequality. And if we fix $\bar\de$ small such that
$C\bar\de^{\frac1m}<\frac12$ (recall that
$p_n\to2^*=\frac{2m}{m-1}$), the above estimate implies
that $\bt^2z_n\to0$ in $H^{\frac12}(\R^m,\mbs_m)$.
Since $y\in\R^m$ and $\bt\in C_c^\infty(\R^m)$ with
$\supp\bt\subset B_1^0(y)$ are arbitrary, the conclusion
follows directly.
\end{proof}

By Lemma \ref{phi-n to phi0}, \eqref{choice of an}
and \eqref{c4}, we have
\[
\lm\int_{B_1^0}|\phi_0|^{2^*}d\vol_{\ig_{\R^m}}=\bar\de.
\]
This implies $\phi_0$ is a non-trivial solution of
\begin{\equ}\label{limit equ}
D_{\ig_{\R^m}}\phi_0=\lm H(a)|\phi_0|^{2^*-2}\phi_0
\quad \text{on } \R^m.
\end{\equ}
By the regularity results (see \cite{Ammann, Isobe-JFA}),
we have $\phi_0\in C^{1,\al}(\R^m,\mbs_m)$ for
some $0<\al<1$.
Since $\phi_0\in L^{2^*}(\R^m,\mbs_m)$ and
$\R^m$ is conformal equivalent to $S^m\setminus\{N\}$
(where $N\in S^m$ is the north pole), it is already known that
$\phi_0$ extends to a non-trivial solution $\bar\phi_0$
to the equation
\[
D\bar\phi_0
=\lm H(a)|\bar\phi_0|^{2^*-2}\bar\phi_0
\quad \text{on } S^m
\]
(cf. \cite[Theorem 5.1]{Ammann2009}, see also
\cite{Ammann}). Recall that
$\frac m2$ is the smallest positive eigenvalue of $D$
on $(S^m,\ig_{S^m})$, and it can be characterized
variationally as (see for instance
\cite{Ammann, AGHM, Ginoux})
\[
\frac{m}2 \om_m^{\frac1m}
=\frac m2 \vol(S^m,\ig_{S^m})^{\frac1m}
=\inf_{\psi}
\frac{\Big(
\int_{S^m}|D\psi|^{\frac{2m}{m+1}}d\vol_{\ig_{S^m}}
\Big)^{\frac{m+1}{m}}}{\int_{S^m}(D\psi,\psi)
d\vol_{\ig_{S^m}}}
\]
where the infimum is taken over the set of all smooth
spinor fields for which
\[
\int_{S^m}(D\psi,\psi)d\vol_{\ig_{S^m}}>0.
\]
Then we can conclude from the conformal transformation that
\begin{\equ}\label{blow-up energy}
\int_{S^m}|\bar\phi_0|^{2^*}d\vol_{\ig_{S^m}}
=\int_{\R^m}|\phi_0|^{2^*}dx
\geq\frac{1}{(\lm H(a))^{m}}\big(\frac m2\big)^m \om_m.
\end{\equ}

With these preparations out of the way, we may now choose
$\eta\in C^\infty(S^m)$ be such that $\eta\equiv1$ on
$B_r(a)$ and $\supp \eta\subset B_{2r}(a)$ for some
$r>0$ (for sure $r$ should not be large in the sense
that we shall assume $3r<inj_{S^m}$
where $inj_{S^m}$ denotes
the injective radius) and define a spinor field $z_n\in
C^\infty(S^m,\mbs(S^m))$ by
\[
z_n=R_n^{-\frac{m-1}2}\eta(\cdot)\ov{(\mu_n)}_*
\circ\phi_0\circ \mu_n^{-1}.
\]
Setting $\va_n=\psi_n-z_n$, we have

\begin{Lem}\label{z-n weakly 0}
$\va_n\rightharpoonup 0$ in $E$ as $n\to\infty$.
\end{Lem}
\begin{proof}
Since we have assumed $\psi_n\rightharpoonup 0$
in $E$ as $n\to\infty$,
we only need to show that $z_n\rightharpoonup0$
in $E$. Remark that, through the conformal transformation
and the local trivialization, it is easy to check that
$\{z_n\}$ is bounded. And hence, by the Sobolev embedding,
this sequence is weakly compact in $E$ and compact
in $L^2$. So, it suffices to prove
\[
\int_{S^m}|z_n|^2d\vol_{\ig_{S^m}}\to0
\]
as $n\to\infty$.

Noting that, for arbitrary $R>0$, we have
\begin{\equ}\label{eee2}
\int_{B_{R_nR}(a_n)}|z_n|^2d\vol_{\ig_{S^m}}
=R_n^{-m+1}\int_{B_{R_nR}^0}|\phi_0|^2d\vol_{\mu_n^*\ig_{S^m}}
=R_n\int_{B_R^0}|\phi_0|^2d\vol_{\ig_n}.
\end{\equ}
And on the other hand, for all large $n$,
\begin{eqnarray*}
\int_{S^m\setminus B_{R_nR}(a_n)}|z_n|^2d\vol_{\ig_{S^m}}
&=&\int_{B_{3r}(a_n)\setminus B_{R_nR}(a_n)}
|z_n|^2d\vol_{\ig_{S^m}} \\
&\leq& C R_n\int_{B_{3r/R_n}^0\setminus B_{R}^0}
|\phi_0|^2 d\vol_{\ig_{\R^m}}  \\
&\leq& C R_n\Big( \int_{B_{3r/R_n}^0\setminus B_{R}^0}
|\phi_0|^{2^*} d\vol_{\ig_{\R^m}} \Big)^{\frac2{2^*}}
\Big(\big(\frac{3r}{R_n}\big)^m-R^m \Big)^\frac{2^*-2}{2^*},
\end{eqnarray*}
where we used $d\vol_{\ig_n}\leq C d\vol_{\ig_{\R^m}}$
on $B_{3r/R_n}^0$ for some constant $C>0$ (since $a_n\to a$
in $S^m$).
Recall that $2^*=\frac{2m}{m-1}$,
it follows from the above inequality that
\begin{\equ}\label{eee3}
\int_{S^m\setminus B_{R_nR}(a_n)}|z_n|^2d\vol_{\ig_{S^m}}
\leq C \Big( \int_{B_{3r/R_n}^0\setminus B_{R}^0}
|\phi_0|^{2^*} d\vol_{\ig_{\R^m}} \Big)^{\frac2{2^*}}
\big( (3r)^m-(R_nR)^m \big).
\end{\equ}

Combining \eqref{eee2} and \eqref{eee3}, we can infer that
\begin{eqnarray*}
\int_{S^m}|z_n|^2d\vol_{\ig_{S^m}}&\leq&
R_n\int_{B_R^0}|\phi_0|^2d\vol_{\ig_n} \\
 & & +C \Big( \int_{B_{3r/R_n}^0\setminus B_{R}^0}
|\phi_0|^{2^*} d\vol_{\ig_{\R^m}} \Big)^{\frac2{2^*}}
\big( (3r)^m-(R_nR)^m \big),
\end{eqnarray*}
which shows $|z_n|_2\to0$ as $n\to\infty$. This completes
the proof.
\end{proof}

Focusing on the description of  the new sequence $\{\va_n\}$,
we have the following result which yields the limiting behavior.

\begin{Lem}\label{va-n converge}
$\cl_{p_n}'(z_n)\to0$ and $\cl_{p_n}'(\va_n)\to0$ as
$n\to\infty$.
\end{Lem}
\begin{proof}
Let $\va\in E$ be an arbitrary test spinor, it follows that
\begin{\equ}\label{dd1}
\cl_{p_n}(z_n)[\va]=\real\int_{S^m}(Dz_n,\va)d\vol_{\ig_{S^m}}
-\real\int_{S^m}H(\xi)|z_n|^{p_n-2}(z_n,\va)d\vol_{\ig_{S^m}}.
\end{\equ}
On the other hand, since $z_n=
R_n^{-\frac{m-1}2}\eta(\cdot)\ov{(\mu_n)}_*
\circ\phi_0\circ \mu_n^{-1}$, we have
\[
Dz_n=R_n^{-\frac{m-1}2}\nabla\eta\cdot_{\ig_{S^m}}
\ov{(\mu_n)}_*\circ\phi_0\circ\mu_n^{-1}+
R_n^{-\frac{m+1}2}\eta(\cdot)\ov{(\mu_n)}_*\circ
(D_{\ig_n}\phi_0)\circ\mu_n^{-1},
\]
where $\cdot_{\ig_{S^m}}$ is the Clifford multiplication
with respect to the metric $\ig_{S^m}$.
Substituting this into \eqref{dd1}, we have
\begin{\equ}\label{dd}
\cl_{p_n}'(z_n)[\va]=l_1+l_2+l_3-l_4
\end{\equ}
where (through the conformal transformation)
\begin{eqnarray*}
l_1&=&R_n^{-\frac{m-1}2}
\real\int_{S^m}\big(\nabla\eta\cdot_{\ig_{S^m}}
\ov{(\mu_n)}_*\circ\phi_0\circ\mu_n^{-1}, \va \big)
d\vol_{\ig_{S^m}}   \\
 &=&R_n^{\frac{m+1}2}\real\int_{B_{3r/R_n}^0}\big(
(\nabla\eta\circ\mu_n)\cdot_{\ig_n}\phi_0, \
 \ov{(\mu_n)}_*^{\,-1}\circ\va\circ\mu_n \big)
 d\vol_{\ig_n},
\end{eqnarray*}
\begin{eqnarray*}
l_2&=&R_n^{-\frac{m+1}2}\real\int_{S^m}\big(
 \eta(\cdot)\ov{(\mu_n)}_*\circ
(D_{\ig_n}\phi_0-D_{\ig_{\R^m}}\phi_0)
\circ\mu_n^{-1},\va\big)
d\vol_{\ig_{S^m}}   \\
&=&R_n^{\frac{m-1}2}\real\int_{B_{3r/R_n}^0}
(\eta\circ\mu_n)\big(D_{\ig_n}\phi_0-D_{\ig_{\R^m}}\phi_0, \
\ov{(\mu_n)}_*^{\,-1}\circ\va\circ\mu_n \big) d\vol_{\ig_n},
\end{eqnarray*}
\begin{eqnarray*}
l_3&=&R_n^{-\frac{m+1}2}\real\int_{S^m}\big(
 \eta(\cdot)\ov{(\mu_n)}_*\circ
(D_{\ig_{\R^m}}\phi_0)\circ\mu_n^{-1}, \va\big)
d\vol_{\ig_{S^m}}   \\
&=&R_n^{\frac{m-1}2}\real\int_{B_{3r/R_n}^0}
(\eta\circ\mu_n)\big(D_{\ig_{\R^m}}\phi_0, \
\ov{(\mu_n)}_*^{\,-1}\circ\va\circ\mu_n \big) d\vol_{\ig_n},
\end{eqnarray*}
and
\begin{eqnarray*}
l_4&=&R_n^{-\frac{m-1}2(p_n-1)}\real\int_{S^m}
H\cdot\eta^{\frac{m+1}{m-1}}\cdot
\big| \ov{(\mu_n)}_*\circ\phi_0\circ\mu_n^{-1} \big|^{p_n-2}
\big( \ov{(\mu_n)}_*\circ\phi_0\circ\mu_n^{-1}, \va\big)
d\vol_{\ig_{S^m}}  \\
&=&R_n^{\frac{m-1}2(2^*+1-p_n)}\real\int_{B_{3r/R_n}^0}
(H\circ\mu_n) (\eta\circ\mu_n)^{\frac{m+1}{m-1}}
|\phi_0|^{p_n-2}(\phi_0,\
\ov{(\mu_n)}_*^{\,-1}\circ\va\circ\mu_n)d\vol_{\ig_n}.
\end{eqnarray*}

We point out here that $l_1$ can be estimated similarly as
we have done in Lemma \ref{z-n weakly 0}. Indeed, by
H\"older inequality, we observe that
\begin{eqnarray}\label{l1}
|l_1|&\leq&  R_n\int_{B_{3r/R_n}^0}\big|(\nabla\eta\circ\mu_n)
\cdot_{\ig_n} \phi_0\big|\cdot\big| R_n^{\frac{m-1}2}
\ov{(\mu_n)}_*^{\,-1}\circ\va\circ\mu_n \big|
d\vol_{\ig_n}  \nonumber \\
&\leq& C R_n\Big( \int_{B_{3r/R_n}^0\setminus
B_{r/2R_n}^0} d\vol_{\ig_{\R^m}} \Big)^{\frac{2^*-2}{2^*}}
\Big( \int_{B_{3r/R_n}^0\setminus
B_{r/2R_n}^0} |\phi_0|^{2^*}d\vol_{\ig_{\R^m}} \Big)^{\frac1{2^*}}
|\va|_{2^*}  \nonumber\\
&\leq& C r^m\Big( \int_{B_{3r/R_n}^0\setminus
B_{r/2R_n}^0} |\phi_0|^{2^*}d\vol_{\ig_{\R^m}} \Big)^{\frac1{2^*}} \|\va\|,
\end{eqnarray}
where we have used the estimate
\[
\int_{B_{3r/R_n}^0}\big| R_n^{\frac{m-1}2}
\ov{(\mu_n)}_*^{\,-1}\circ\va\circ\mu_n \big|^{2^*}
d\vol_{\ig_n}=\int_{B_{3r}(a_n)}|\va|^{2^*}
d\vol_{\ig_{S^m}}\leq |\va|_{2^*}^{2^*}.
\]
Since $\phi_0\in L^{2^*}(\R^m,\mbs_m)$, we obtain
from \eqref{l1}
\begin{\equ}\label{l11}
|l_1|\leq o_n(1) \|\va\| \quad \text{as } n\to\infty.
\end{\equ}

For $l_2$, by H\"older inequality again, we have
\begin{eqnarray}\label{l2}
|l_2|&\leq&\int_{B_{3r/R_n}^0}\big| D_{\ig_n}\phi_0
-D_{\ig_{\R^m}}\phi_0 \big| \cdot
\big| R_n^{\frac{m-1}2}
\ov{(\mu_n)}_*^{\,-1}\circ\va\circ\mu_n \big|
d\vol_{\ig_n}  \nonumber \\
&\leq& C\Big( \int_{B_{3r/R_n}^0}\big| D_{\ig_n}\phi_0
-D_{\ig_{\R^m}}\phi_0 \big|^{\frac{2m}{m+1}}
d\vol_{\ig_{\R^m}} \Big)^{\frac{m+1}{2m}} \|\va\|.
\end{eqnarray}
Fix $R>0$ arbitrarily, we deduce that
\[
\aligned
&\int_{B_{3r/R_n}^0}\big| D_{\ig_n}\phi_0
-D_{\ig_{\R^m}}\phi_0 \big|^{\frac{2m}{m+1}}
d\vol_{\ig_{\R^m}} \\
&\qquad=\int_{B_{3r/R_n}^0\setminus B_R^0}
\big| D_{\ig_n}\phi_0
-D_{\ig_{\R^m}}\phi_0 \big|^{\frac{2m}{m+1}}
d\vol_{\ig_{\R^m}}
 +\int_{B_R^0}
\big| D_{\ig_n}\phi_0
-D_{\ig_{\R^m}}\phi_0 \big|^{\frac{2m}{m+1}}
d\vol_{\ig_{\R^m}}
\endaligned
\]
and, since $\nabla\phi_0\in L^{\frac{2m}{m+1}}
(\R^m,\mbs_m)$ and $\ig_n\to\ig_{\R^m}$
in $C^\infty(B_R^0)$ as $n\to\infty$, we can
get further from \eqref{l2} that
\begin{\equ}\label{l21}
|l_2|\leq o_n(1)\|\va\| \quad \text{as } n\to\infty.
\end{\equ}

Now, it remains to estimate $|l_3-l_4|$. Noting that
$\phi_0$ satisfies Eq. \eqref{limit equ}, we soon have
\begin{\equ}\label{l3}
l_3=\lm H(a)\,\real\int_{B_{3r/R_n}^0}(\eta\circ\mu_n)
|\phi_0|^{2^*-2}\big( \phi_0,\  R_n^{\frac{m-1}2}
\ov{(\mu_n)}_*^{\,-1}\circ\va\circ\mu_n\big)
d\vol_{\ig_n}.
\end{\equ}
Since the "blow-up points" $a_n\to a$ in $S^m$ and
$\eta\equiv1$ on $B_r(a)$, we have
$\eta\circ\mu_n\equiv1$ on $B_R^0$ for all large $n$
where $R>0$ is fixed. Therefore, by
$\displaystyle\lim_{n\to\infty}
R_n^{\frac{m-1}2(2^*-p_n)}=\lm$ and
$H\circ\mu_n\to H(a)$ uniformly on $B_R^0$ as
$n\to\infty$, we have
\begin{\equ}\label{l31}
\aligned
&R_n^{\frac{m-1}2(2^*-p_n)}\real\int_{B_R^0}
(H\circ\mu_n)|\phi_0|^{p_n-2}\big( \phi_0,\
R_n^{\frac{m-1}2}
\ov{(\mu_n)}_*^{\,-1}\circ\va\circ\mu_n \big)
d\vol_{\ig_n} \\
&\qquad =\lm H(a)
\real\int_{B_R^0}
|\phi_0|^{2^*-2}\big( \phi_0,\
R_n^{\frac{m-1}2}
\ov{(\mu_n)}_*^{\,-1}\circ\va\circ\mu_n \big)
d\vol_{\ig_n} +o_n(1)\|\va\|.
\endaligned
\end{\equ}
On the other hand, since $\phi_0\in L^{2^*}(\R^m,\mbs)$,
it follows that
\[
\aligned
&\int_{B_{3r/R_n}^0\setminus B_R^0}(\eta\circ\mu_n)
|\phi_0|^{2^*-1}\cdot \big| R_n^{\frac{m-1}2}
\ov{(\mu_n)}_*^{\,-1}\circ\va\circ\mu_n\big|
d\vol_{\ig_n} \\
&\qquad
\leq C \Big( \int_{B_{3r/R_n}^0\setminus B_R^0}
|\phi_0|^{2^*} d\vol_{\ig_{\R^m}} \Big)^{\frac{2^*-1}
{2^*}}\Big(\int_{B_{3r}(a_n)}|\va|^{2^*}d\vol_{\ig_{S^m}}
\Big)^{\frac1{2^*}}
\endaligned
\]
and similarly
\[
\aligned
&R_n^{\frac{m-1}2(2^*-p_n)}
\int_{B_{3r/R_n}^0\setminus B_R^0}(H\circ\mu_n)
(\eta\circ\mu_n)^{\frac{m+1}{m-1}}|\phi|^{p_n-1}
\cdot \big| R_n^{\frac{m-1}2}
\ov{(\mu_n)}_*^{\,-1}\circ\va\circ\mu_n\big|
d\vol_{\ig_n} \\
&\qquad\leq C R_n^{\frac{m-1}2(2^*-p_n)}
\Big( \int_{B_{3r/R_n}^0\setminus B_R^0}
d\vol_{\ig_{\R^m}} \Big)^{\frac{2^*-p_n}{2^*}}
\Big(\int_{B_{3r/R_n}^0\setminus B_R^0}
|\phi_0|^{2^*} d\vol_{\ig_{\R^m}}
\Big)^{\frac{p_n-1}{2^*}} \cdot|\va|_{2^*} \\[0.4em]
&\qquad\leq C \big( (3r)^m-(R_nR)^m
\big)^{\frac{2^*-p_n}{2^*}}
\Big(\int_{B_{3r/R_n}^0\setminus B_R^0}
|\phi_0|^{2^*} d\vol_{\ig_{\R^m}}
\Big)^{\frac{p_n-1}{2^*}} \cdot\|\va\|.
\endaligned
\]
Thus, combining \eqref{l3}, \eqref{l31} and the above
two estimates, we can conclude
\begin{\equ}\label{l3-l4}
|l_3-l_4|\leq o_n(1)\|\va\| \quad \text{as } n\to\infty.
\end{\equ}
And then, it follows from \eqref{l11}, \eqref{l21} and
\eqref{l3-l4} that $\cl_{p_n}'(z_n)\to0$ as $n\to\infty$.

\medskip

Now we turn to prove $\cl_{p_n}'(\va_n)\to0$
as $n\to\infty$.

Again, we choose $\va\in E$ be an arbitrary test spinor.
We then have
\begin{eqnarray}
\cl_{p_n}'(\va_n)[\va]&=&\real\int_{S^m}(D\va_n,\va)
d\vol_{\ig_{S^m}}-\real\int_{S^m}H(\xi)|\va_n|^{p_n-2}
(\va_n,\va)d\vol_{\ig_{S^m}}  \nonumber\\
&=&\cl_{p_n}'(\psi_n)[\va]-\cl_{p_n}'(z_n)[\va]
+\real\int_{S^m}(\Psi_n,\va)d\vol_{\ig_{S^m}},
\end{eqnarray}
where
\[
\Psi_n=H(\xi)|\psi_n|^{p_n-2}\psi_n-H(\xi)|z_n|^{p_n-2}
z_n-H(\xi)|\va_n|^{p_n-2}\va_n.
\]

Since we have assumed $\{\psi_n\}$ satisfies
\eqref{key-assumption}, it follows that we only
need to show that $\|\Psi_n\|_{E^*}\to0$ as
$n\to\infty$. Similarly as was argued in \eqref{integral1},
we will use the fact that there exists $C>0$
(independent of $n$) such that
\[
|\Psi_n|\leq C|z_n|^{p_n-2}|\va_n|
+C|\va_n|^{p_n-2}|z_n|.
\]
For any $R>0$, we first observe that for all $n$ large
\[
\aligned
&\int_{S^m\setminus B_{R_nR}(a_n)}|z_n|^{p_n-2}
\cdot|\va_n|\cdot|\va| d\vol_{\ig_{S^m}} \\
&\qquad \leq \om_m^{\frac{2^*-p_n}{2^*}}\Big(
\int_{S^m\setminus B_{R_nR}(a_n)}|z_n|^{2^*}
d\vol_{\ig_{S^m}}\Big)^{\frac{p_n-2}{2^*}} \Big(
\int_{S^m\setminus B_{R_nR}(a_n)}|\va_n|^{2^*}
d\vol_{\ig_{S^m}}\Big)^{\frac1{2^*}} |\va|_{2^*} \\
&\qquad \leq C \Big(
\int_{B_{3r/R_n}^0\setminus B_R^0}|\phi_0|^{2^*}
d\vol_{\ig_n} \Big)^{\frac{p_n-2}{2^*}}\|\va_n\|
\cdot \|\va\| = o_R(1)\|\va\|,
\endaligned
\]
and
\[
\aligned
&\int_{S^m\setminus B_{R_nR}(a_n)}|\va_n|^{p_n-2}
\cdot|z_n|\cdot|\va| d\vol_{\ig_{S^m}} \\
&\qquad \leq \om_m^{\frac{2^*-p_n}{2^*}}\Big(
\int_{S^m\setminus B_{R_nR}(a_n)}|\va_n|^{2^*}
d\vol_{\ig_{S^m}}\Big)^{\frac{p_n-2}{2^*}} \Big(
\int_{S^m\setminus B_{R_nR}(a_n)}|z_n|^{2^*}
d\vol_{\ig_{S^m}}\Big)^{\frac1{2^*}} |\va|_{2^*} \\
&\qquad \leq C \Big(
\int_{B_{3r/R_n}^0\setminus B_R^0}|\phi_0|^{2^*}
d\vol_{\ig_n} \Big)^{\frac1{2^*}}\|\va_n\|^{p_n-2}
\cdot \|\va\| = o_R(1)\|\va\|,
\endaligned
\]
where $\om_m$ stands for the volume of
$(S^m,\ig_{S^m})$ and $o_R(1)\to0$ as $R\to\infty$.

And on the other hand, inside $B_{R_nR}(a_n)$, we have
\[
\aligned
&\int_{B_{R_nR}(a_n)}|z_n|^{p_n-2}
\cdot|\va_n|\cdot|\va| d\vol_{\ig_{S^m}} \\
&\qquad \leq \om_m^{\frac{2^*-p_n}{2^*}}\Big(
\int_{B_{R_nR}(a_n)}|z_n|^{2^*}
d\vol_{\ig_{S^m}}\Big)^{\frac{p_n-2}{2^*}} \Big(
\int_{B_{R_nR}(a_n)}|\va_n|^{2^*}
d\vol_{\ig_{S^m}}\Big)^{\frac1{2^*}} |\va|_{2^*} \\
&\qquad \leq C \Big(
\int_{\R^m}|\phi_0|^{2^*} d\vol_{\ig_{\R^m}}
\Big)^{\frac{p_n-2}{2^*}} \Big(
\int_{B_R^0}|\phi_n-\phi_0|^{2^*}
d\vol_{\ig_n} \Big)^{\frac1{2^*}}
\cdot \|\va\| = o_n(1)\|\va\|
\endaligned
\]
and
\[
\aligned
&\int_{B_{R_nR}(a_n)}|\va_n|^{p_n-2}
\cdot|z_n|\cdot|\va| d\vol_{\ig_{S^m}} \\
&\qquad \leq \om_m^{\frac{2^*-p_n}{2^*}}\Big(
\int_{B_{R_nR}(a_n)}|\va_n|^{2^*}
d\vol_{\ig_{S^m}}\Big)^{\frac{p_n-2}{2^*}} \Big(
\int_{B_{R_nR}(a_n)}|z_n|^{2^*}
d\vol_{\ig_{S^m}}\Big)^{\frac1{2^*}} |\va|_{2^*} \\
&\qquad \leq C \Big(
\int_{B_R^0}|\phi_n-\phi_0|^{2^*} d\vol_{\ig_{\R^m}}
\Big)^{\frac{p_n-2}{2^*}} \Big(
\int_{\R^m}|\phi_0|^{2^*}
d\vol_{\ig_n} \Big)^{\frac1{2^*}}
\cdot \|\va\| = o_n(1)\|\va\|
\endaligned
\]
as $n\to\infty$, where we have used the fact
$\phi_n\to\phi_0$ in $H_{loc}^{\frac12}(\R^m,\mbs_m)$
(see Lemma \ref{phi-n to phi0}).
Therefore, we can conclude that $\Psi_n\to0$ in $E^*$
as $n\to\infty$ which completes the proof.
\end{proof}

At this point we have the following result which summarizes
the blow-up phenomenon.

\begin{Prop}\label{blow-up prop}
Let $\{\psi_n\}\subset E$ fulfill the assumption of
Proposition \ref{alternative prop}. If $\{\psi_n\}$
does not contain any compact subsequence.
Then, up to a subsequence
if necessary, there exist a convergent sequence
$\{a_n\}\subset S^m$, $a_n\to a$ as $n\to\infty$, a
sequence of radius $\{R_n\}$ converging to
$0$, a real number $\lm\in\big(2^{-\frac1{m-1}},1\big]$
and a non-trivial solution $\phi_0$ of
Eq. \eqref{limit equ} such that
\[
R_n^{\frac{m-1}2(2^*-p_n)}=\lm+o_n(1)
\]
and
\[
\psi_n=R_n^{-\frac{m-1}2}\eta(\cdot)\ov{(\mu_n)}_*
\circ\phi_0\circ \mu_n^{-1}+o_n(1) \quad \text{in } E
\]
as $n\to\infty$,
where $\mu_n(x)=\exp_{a_n}(R_nx)$
and $\eta\in C^\infty(S^m)$ is a cut-off function
such that $\eta(\xi)=1$ on $B_r(a)$ and
$\supp\eta\subset B_{2r}(a)$, some $r>0$. Moreover, we
have
\[
\cl_{p_n}(\psi_n)\geq\frac{1}{2m(\lm H(a))^{m-1}}
\big(\frac m2\big)^m \om_m+o_n(1)
\]
as $n\to\infty$.
\end{Prop}
\begin{proof}
Inherit from the previous lemmas, let us first set
$z_n=R_n^{-\frac{m-1}2}\eta(\cdot)\ov{(\mu_n)}_*
\circ\phi_0\circ \mu_n^{-1}$ and $\va_n=\psi_n-z_n$.
By Lemma \ref{va-n converge}, we have
$\cl_{p_n}'(z_n)\to0$ and $\cl_{p_n}'(\va_n)\to0$
as $n\to\infty$. Hence,
\[
\cl_{p_n}(z_n)+o_n(1)=
\cl_{p_n}(z_n)-\frac12\cl_{p_n}'(z_n)[z_n]=
\frac{p_n-2}{2p_n}\int_{S^m}H(\xi)|z_n|^{p_n}
d\vol_{\ig_{S^m}}\geq0
\]
and
\[
\cl_{p_n}(\va_n)+o_n(1)=
\cl_{p_n}(\va_n)-\frac12\cl_{p_n}'(\va_n)[\va_n]=
\frac{p_n-2}{2p_n}\int_{S^m}H(\xi)|\va_n|^{p_n}
d\vol_{\ig_{S^m}}\geq0
\]
We claim that
\begin{claim}\label{claim1}
$\cl_{p_n}(\psi_n)=\cl_{p_n}(z_n)+\cl_{p_n}(\va_n)
+o_n(1)$ as $n\to\infty$.
\end{claim}
\noindent
Assuming Claim \ref{claim1} for the moment, then we shall
get $\cl_{p_n}(\va_n)\to0$ as $n\to\infty$. Indeed,
suppose to the contrary that (up to a subsequence)
$\cl_{p_n}(\va_n)\geq c>0$, it follows from the
boundedness of $\{\va_n\}$ in $E$,
Corollary \ref{key2} and Lemma \ref{msf=I} that
\[
\tau_{p_n}\leq\msf_{p_n}(\va_n^+)=\max_{t>0}\Big(
\frac{2p_n}{p_n-2}I_{p_n}(t\va_n^+)
\Big)^{\frac{p_n-2}{p_n}}\leq\Big(\frac{2p_n}{p_n-2}
\cl_{p_n}(\va_n)+o_n(1) \Big)^{\frac{p_n-2}{p_n}}.
\]
Hence, by the left continuity of $p\mapsto \tau_p$
(see Proposition \ref{key3}), we get
\begin{\equ}\label{blow-up-cl1}
\cl_{p_n}(\va_n)\geq\frac{p_n-2}{2p_n}
(\tau_{p_n})^{\frac{p_n}{p_n-2}}
+o_n(1)=\frac1{2m}(\tau_{2^*})^m+o_n(1).
\end{\equ}
On the other hand,  we have
\begin{eqnarray*}
\cl_{p_n}(z_n)&=&\frac{p_n-2}{2p_n}
\int_{S^m}H(\xi)|z_n|^{p_n}d\vol_{\ig_{S^m}}+o_n(1)
\\[0.4em]
&=&\frac{p_n-2}{2p_n} R_n^{\frac{m-1}2(2^*-p_n)}
\int_{B_R^0}(H\circ\mu_n)|\phi_0|^{p_n} d\vol_{\ig_n}
+o_n(1)+o_R(1)\\[0.4em]
&=&\frac1{2m} \lm H(a)\int_{B_R^0}|\phi_0|^{2^*}
d\vol_{\ig_{\R^m}} + o_n(1) + o_R(1)
\end{eqnarray*}
for $R>0$ large. Thus, by \eqref{tau value},
\eqref{blow-up energy},
$\lm\in(0,1]$ and $H(a)\leq H_{max}$, we obtain
\begin{\equ}\label{blow-up-cl2}
\cl_{p_n}(z_n)\geq \frac1{2m(\lm H(a))^{m-1}}
\big(\frac m2\big)^m\om_m
+o_n(1)\geq\frac1{2m}(\tau_{2^*})^m+o_n(1).
\end{\equ}
Combining Claim \ref{claim1}, \eqref{blow-up-cl1} and
\eqref{blow-up-cl2}, we have $\cl_{p_n}(\psi_n)
\geq\frac1m(\tau_{2^*})^m+o_n(1)$ as $n\to\infty$
which contradicts to \eqref{key-assumption}.
Therefore, we have $\cl_{p_n}(\va_n)\to0$ as $n\to\infty$
and this, together with $\cl_{p_n}'(\va_n)\to0$, implies
$\va_n\to0$ in $E$ as $n\to\infty$. Moreover, we can
get a lower bound for $\lm$ since $H(a)\leq H_{max}$
and $\cl_{p_n}(\psi_n)<\frac1m(\tau_{2^*})^m$, i.e.
$\lm>2^{-\frac1{m-1}}$.

\medskip

Now it remains to prove Claim \ref{claim1}. We would like
to point out here that (thanks to Lemma \ref{va-n converge})
 this is equivalent to show
\begin{\equ}\label{claim1-1}
\int_{S^m}(D\psi_n,\psi_n)d\vol_{\ig_{S^m}}=
\int_{S^m}(Dz_n,z_n)d\vol_{\ig_{S^m}}+
\int_{S^m}(D\va_n,\va_n)d\vol_{\ig_{S^m}}+o_n(1).
\end{\equ}
And since $\va_n=\psi_n-z_n$, it suffices to prove
$\int_{S^m}(Dz_n,\va_n)d\vol_{\ig_{S^m}}=o_n(1)$
as $n\to\infty$. In fact, for arbitrary $R>0$, we have
\[
\aligned
\int_{S^m}(Dz_n,\va_n)d\vol_{\ig_{S^m}}
&=\int_{B_{R_nR}(a_n)}(Dz_n,\va_n)d\vol_{\ig_{S^m}}
+\int_{B_{3r}(a_n)\setminus B_{R_nR}(a_n)}
(Dz_n,\va_n)d\vol_{\ig_{S^m}}  \\[0.4em]
&=\int_{B_R^0}(D_{\ig_n}\phi_0,\phi_n-\phi_0)
d\vol_{\ig_n}+\int_{B_{3r/R_n}^0\setminus B_R^0}
(D_{\ig_n}\phi_0,\phi_n-\phi_0) d\vol_{\ig_n}.
\endaligned
\]
And for the first integral, by Lemma \ref{phi-n to phi0},
we can get
\begin{\equ}\label{blow-up-cl3}
\Big|\int_{B_R^0}(D_{\ig_n}\phi_0,\phi_n-\phi_0)
d\vol_{\ig_n}\Big|\leq C|\nabla\phi_0|_{\frac{2m}
{m+1}} \cdot \|\phi_n-\phi_0\|_{H^{1/2}_{loc}}\to0
\end{\equ}
as $n\to\infty$. Meanwhile to estimate the second integral,
we first observe that (through the conformal transformation)
\[
\sup_{n}\int_{B_{3r/R_n}^0}|\phi_n-\phi_0|^{2^*}
d\vol_{\ig_{\R^m}}\leq C\sup_n\int_{B_{3r}(a_n)}
|\psi_n-z_n|^{2^*}d\vol_{\ig_{S^m}}<+\infty
\]
for some $C>0$. Thus, by $d\vol_{\ig_n}\leq
C d\vol_{\ig_{\R^m}}$, we have
\begin{\equ}\label{blow-up-cl4}
\Big|\int_{B_{3r/R_n}^0\setminus B_R^0}
(D_{\ig_n}\phi_0,\phi_n-\phi_0) d\vol_{\ig_n}\Big|
\leq C\Big( \int_{B_{3r/R_n}^0\setminus B_R^0}
|\nabla \phi_0|^{\frac{2m}{m+1}}d\vol_{\ig_{\R^m}}
\Big)^{\frac{m+1}{2m}}\to0
\end{\equ}
as $R\to\infty$. Therefore, by \eqref{blow-up-cl3}
and \eqref{blow-up-cl4}, we obtain \eqref{claim1-1}
is valid and the proof is hereby completed.
\end{proof}

\subsection{Using the stereographic projection}\label{stereo sec}

According to Proposition \ref{blow-up prop}: any
non-compact sequence $\{\psi_n\}$ which satisfies
\eqref{key-assumption}, blows up around a point
$a\in S^m$. And due to the statement, it is natural to ask further questions:
\begin{itemize}
\item[1.] Where the blow-up point $a$ locates or whether $a$ has any relation with the function $H$ particularly when $H\not\equiv constant$?

\item[2.] Whether or not the value of $\lm$ can be fixed precisely?
\end{itemize}

We will show now that, if blow-up happens,
such $a\in S^m$ must be a critical point of $H$ and $\lm\equiv1$.
Before proving the results, we begin with some elementary
materials on stereographic projection.

First of all, for arbitrary $\xi\in S^m$, we can always embed
$S^m$ into $\R^{m+1}$ in the way that $\xi$ has the
coordinate $\xi=(0,\dots,0,-1)\in\R^{m+1}$, i.e. $\xi$ is the
South pole.
Denoting $\cs_\xi: S^m\setminus\{-\xi\}\to\R^m$
the stereographic projection from the new North pole $-\xi$,
we have $\cs_\xi(\xi)=0$. Moreover,
$S^m\setminus\{-\xi\}$ and $\R^m$ are conformally
equivalent due to the fact
$(\cs_\xi^{-1})^*\ig_{S^m}=f^2\ig_{\R^m}$ with
$f(x)=\frac2{1+|x|^2}$.

Recall the conformal transformation formula mentioned in
Proposition \ref{conformal formula}, there is an
isomorphism of vector bundles
$\iota:\mbs\big(\R^m, (\cs_\xi^{-1})^*\ig_{S^m}\big)\to
\mbs(\R^m,\ig_{\R^m})$ such that
\[
D_{\ig_{\R^m}}\big( \iota(\va) \big) = \iota\big( f^{\frac{m+1}2}
D_{(\cs_\xi^{-1})^*\ig_{S^m}} (f^{-\frac{m-1}2}\va)\big),
\]
where $D_{(\cs_\xi^{-1})^*\ig_{S^m}}$ is the Dirac
operator on $\R^m$ with respect to the metric
$(\cs_\xi^{-1})^*\ig_{S^m}$. Thus when $\psi$ is a
solution to the equation $D\psi=H(\xi)|\psi|^{p-2}\psi$
on $(S^m,\ig_{S^m})$ for some $p\in(2,2^*]$, then
$\phi:=\iota(f^{\frac{m-1}2}\psi\circ\cs_\xi^{-1})$ will
satisfies the transformed equation
\[
D_{\ig_{\R^m}}\phi=f^{\frac{m-1}2(2^*-p)}
(H\circ \cs_{\xi}^{-1})|\phi|^{p-2}\phi \quad
\text{on } (\R^m,\ig_{\R^m}).
\]
Moreover, since $d\vol_{(\cs_\xi^{-1})^*\ig_{S^m}}
=f^{m}d\vol_{\ig_{\R^m}}$, we have
\[
\int_{\R^m}(D_{\ig_{\R^m}}\phi,\phi)d\vol_{\ig_{\R^m}}
=\int_{S^m}(D\psi,\psi) d\vol_{\ig_{S^m}},
\]
\[
\int_{\R^m}f^{\frac{m-1}2(2^*-p)}|\phi|^{p}
d\vol_{\ig_{\R^m}}=\int_{S^m}|\psi|^{p}
d\vol_{\ig_{S^m}},
\]
and
\[
\int_{\R^m}|\phi|^{2^*}d\vol_{\ig_{\R^m}}=
\int_{S^m}|\psi|^{2^*}d\vol_{\ig_{S^m}}.
\]

\medskip

Returning to our case, let us assume $\{\psi_n\}\subset E$
be a sequence of solutions to the equations
\begin{\equ}\label{equs-n}
D\psi_n=H(\xi)|\psi_n|^{p_n-2}\psi_n \quad \text{on }
S^m, \quad n=1,2,\dots
\end{\equ}
and satisfying
\begin{\equ}\label{key-assumption1}
\frac1{2m}(\tau_{2^*})^m\leq\cl_{p_n}(\psi_n)
\leq\frac1m(\tau_{2^*})^m-\theta
\end{\equ}
for all $n$ large and some $\theta>0$. Then, it is
clear that $\cl_{p_n}'(\psi_n)\equiv 0$ for all $n$. And
hence $\{\psi_n\}$ fulfills the assumption of
Proposition \ref{alternative prop}. Moreover,
by the regularity results proved in \cite{Ammann}, these
solutions are in fact $C^{1,\al}$ for some $\al\in(0,1)$
and are classical solutions to \eqref{equs-n}.

\begin{Prop}\label{blow-up prop2}
Suppose $\{\psi_n\}$ satisfies \eqref{equs-n} and
\eqref{key-assumption1} and does not contain any compact
subsequence. Let $a\in S^m$ be the associate
blow-up point found in Proposition \ref{blow-up prop}
(up to a subsequence if necessary). Then $\nabla H(a)=0$.
\end{Prop}
\begin{proof}
Let us consider the stereographic projection
$\cs_a: S^m\setminus\{N\}\to\R^m$ and the
associated bundle isomorphism
$\iota:\mbs\big(\R^m, (\cs_a^{-1})^*\ig_{S^m}\big)\to
\mbs(\R^m,\ig_{\R^m})$. Denoted by
$\tilde\phi_n= \iota(f^{\frac{m-1}2}\psi_n\circ \cs_a^{-1})$,
we have $\tilde\phi_n$ satisfies
\begin{\equ}\label{tilde phi-n}
D_{\ig_{\R^m}}\tilde\phi_n=f^{\frac{m-1}2(2^*-p_n)}
(H\circ \cs_a^{-1})|\tilde\phi_n|^{p_n-2}\tilde\phi_n
\quad \text{on } (\R^m,\ig_{\R^m}).
\end{\equ}

Take $\bt\in C_c^\infty(S^m)$ be a cut-off function on
$S^m$ such that $\bt\equiv1$ on $B_{2r}(a)$ and
$\supp\bt\subset B_{3r}(a)$ where $r>0$ comes from
Proposition \ref{blow-up prop}. Then we are allowed
to multiply \eqref{tilde phi-n} by
$\pa_k\big((\bt\circ \cs_a^{-1})\tilde\phi_n \big)$
as a test spinor for each $k=1,2,\dots,m$, and
consequently we have
\begin{\equ}\label{identity}
\aligned
&\real\int_{\R^m}\big(D_{\ig_{\R^m}}\tilde\phi_n,
\pa_k\big((\bt\circ \cs_a^{-1})\tilde\phi_n \big) \big)
d\vol_{\ig_{\R^m}} \\
&\qquad=\real\int_{\R^m}
f^{\frac{m-1}2(2^*-p_n)}(H\circ\cs_a^{-1})
|\tilde\phi_n|^{p_n-2}\big(\tilde\phi_n,  \,
\pa_k\big((\bt\circ \cs_a^{-1})\tilde\phi_n \big)   \big)
d\vol_{\ig_{\R^m}}.
\endaligned
\end{\equ}

Remark that $(\bt\circ \cs_a^{-1})\tilde\phi_n$ has a
compact support, we may integrate by parts to get
\begin{eqnarray}\label{I1}
0&=&\real\int_{\R^m}\pa_k\big(
D_{\ig_{\R^m}}\tilde\phi_n,
(\bt\circ\cs_a^{-1})\tilde\phi_n \big) d\vol_{\ig_{\R^m}}
\nonumber \\
&=&2\,\real\int_{\R^m}\big(D_{\ig_{\R^m}}\tilde\phi_n,
\,\pa_k\big((\bt\circ \cs_a^{-1})\tilde\phi_n \big) \big)
d\vol_{\ig_{\R^m}}  \nonumber\\
& &\quad +\,\real\int_{\R^m}\big(\pa_k\tilde\phi_n,
\nabla(\bt\circ\cs_a^{-1})\cdot_{\ig_{\R^m}}\tilde\phi_n
\big) d\vol_{\ig_{\R^m}}  \nonumber\\
& &\quad -\,\real\int_{\R^m}\big( D_{\ig_{\R^m}}\tilde\phi_n,  \,
\pa_k(\bt\circ\cs_a^{-1})\tilde\phi_n \big)
d\vol_{\ig_{\R^m}},
\end{eqnarray}
where $\cdot_{\ig_{\R^m}}$ denotes the Clifford
multiplication with respect to $\ig_{\R^m}$.
Now let us evaluate the last two integrals of the
previous equality. First of all, by noting that
$\{\psi_n\}$ is bounded in $E$, we can see
from the conformal transformation and the regularity
results (see \cite{Ammann}) that
$\{\nabla\tilde\phi_n\}$ is uniformly bounded in $L^{\frac{2m}{m+1}}(\R^m,\mbs_m)$.
And so, by Proposition \ref{blow-up prop},
\[
\Big| \int_{\R^m}\big(\pa_k\tilde\phi_n,
\nabla(\bt\circ\cs_a^{-1})\cdot_{\ig_{\R^m}}\tilde\phi_n
\big) d\vol_{\ig_{\R^m}} \Big|
\leq
C\Big( \int_{B_{3r}(a)\setminus B_{2r}(a)}
 |\psi_n|^{2^*}
d\vol_{\ig_{S^m}} \Big)^{\frac1{2^*}}  \to0
\]
as $n\to\infty$. Analogously, we have
\[
\Big| \int_{\R^m}\big( D_{\ig_{\R^m}}\tilde\phi_n,  \,
\pa_k(\bt\circ\cs_a^{-1})\tilde\phi_n \big)
d\vol_{\ig_{\R^m}} \Big|\to0
\]
as $n\to\infty$. And thus, we conclude from \eqref{I1}
that
\begin{\equ}\label{I2}
\real\int_{\R^m}\big(D_{\ig_{\R^m}}\tilde\phi_n,
\,\pa_k\big((\bt\circ \cs_a^{-1})\tilde\phi_n \big) \big)
d\vol_{\ig_{\R^m}}=o_n(1) \quad \text{as } n\to\infty.
\end{\equ}

On the other hand, to evaluate the second integral
of \eqref{identity}, we have
\begin{eqnarray}\label{I3}
0&=&\int_{\R^m}\pa_k\big[ f^{\frac{m-1}2(2^*-p_n)}
(H\circ\cs_a^{-1})(\bt\circ\cs_a^{-1})|\tilde\phi_n|^{p_n}
\big]d\vol_{\ig_{\R^m}}  \nonumber\\
&=&\frac{m-1}2(2^*-p_n)\int_{\R^m}
f^{\frac{m-1}2(2^*-p_n)-1}\pa_kf \cdot
(H\circ\cs_a^{-1})(\bt\circ\cs_a^{-1})|\tilde\phi_n|^{p_n}
d\vol_{\ig_{\R^m}}  \nonumber\\
& &\quad+\int_{\R^m}f^{\frac{m-1}2(2^*-p_n)}
\pa_k(H\circ\cs_a^{-1})(\bt\circ\cs_a^{-1})
|\tilde\phi_n|^{p_n} d\vol_{\ig_{\R^m}}  \nonumber\\
& &\quad +\,p_n\real\int_{\R^m}
f^{\frac{m-1}2(2^*-p_n)}(H\circ\cs_a^{-1})
|\tilde\phi_n|^{p_n-2}\big(\tilde\phi_n,  \,
\pa_k\big((\bt\circ \cs_a^{-1})\tilde\phi_n \big)   \big)
d\vol_{\ig_{\R^m}}   \nonumber\\
& &\quad -\,(p_n-1)\int_{\R^m}
f^{\frac{m-1}2(2^*-p_n)}(H\circ\cs_a^{-1})\pa_k
(\bt\circ\cs_a^{-1})|\tilde\phi_n|^{p_n}
d\vol_{\ig_{\R^m}}.
\end{eqnarray}
It is evident that the last integral converges to $0$ as
$n\to\infty$, and we only need to estimate the remaining
terms. Notice that $f(x)=\frac2{1+|x|^2}$
and $\bt\circ\cs_a^{-1}$ has a compact
support on $\R^m$, hence $f$, $f^{-1}$ and $\nabla f$
are bounded uniformly on $\supp(\bt\circ\cs_a^{-1})$ and
\[
\Big|\frac{m-1}2(2^*-p_n) \int_{\R^m}
f^{\frac{m-1}2(2^*-p_n)-1}\pa_kf \cdot
(H\circ\cs_a^{-1})(\bt\circ\cs_a^{-1})|\tilde\phi_n|^{p_n}
d\vol_{\ig_{\R^m}} \Big| \to0
\]
as $n\to\infty$. For the second integral, take arbitrarily
$R>0$ small, we deduce that
\[
\aligned
&\Big|
\int_{\R^m\setminus B_{R}^0}f^{\frac{m-1}2(2^*-p_n)}
\pa_k(H\circ\cs_a^{-1})(\bt\circ\cs_a^{-1})
|\tilde\phi_n|^{p_n} d\vol_{\ig_{\R^m}}
\Big| \\
&\qquad \leq C \int_{\R^m\setminus B_{R}^0}
f^{\frac{m-1}2(2^*-p_n)}|\tilde\phi_n|^{p_n}
d\vol_{\ig_{\R^m}}
\leq C \int_{S^m\setminus B_R(a)}|\psi_n|^{p_n}
d\vol_{\ig_{S^m}}\to0
\endaligned
\]
as $n\to\infty$. And inside $B_R^0$, we have
\[
\aligned
&\int_{B_R^0}f^{\frac{m-1}2(2^*-p_n)}
\pa_k(H\circ\cs_a^{-1})(\bt\circ\cs_a^{-1})
|\tilde\phi_n|^{p_n} d\vol_{\ig_{\R^m}}  \\
&\qquad = \pa_k(H\circ\cs_a^{-1})(0)\int_{B_R^0}
f^{\frac{m-1}2(2^*-p_n)}
|\tilde\phi_n|^{p_n} d\vol_{\ig_{\R^m}} +O\Big(
\int_{B_R^0}|x|\cdot|\tilde\phi_n|^{p_n} d\vol_{\ig_{\R^m}}\Big) + o_n(1) \\
&\qquad = \pa_k(H\circ\cs_a^{-1})(0)\int_{B_R^0}
f^{\frac{m-1}2(2^*-p_n)}
|\tilde\phi_n|^{p_n} d\vol_{\ig_{\R^m}} +O(R) + o_n(1)
\endaligned
\]
as $n\to\infty$ and $R\to0$. Thus by \eqref{I3},
 for arbitrarily small $R>0$, we get
\begin{\equ}\label{I4}
\aligned
&\real\int_{\R^m}
f^{\frac{m-1}2(2^*-p_n)}(H\circ\cs_a^{-1})
|\tilde\phi_n|^{p_n-2}\big(\tilde\phi_n,  \,
\pa_k\big((\bt\circ \cs_a^{-1})\tilde\phi_n \big)   \big)
d\vol_{\ig_{\R^m}}  \\
&\qquad =
-\frac1{p_n}\pa_k(H\circ\cs_a^{-1})(0)\int_{\R^m}
f^{\frac{m-1}2(2^*-p_n)}
|\tilde\phi_n|^{p_n} d\vol_{\ig_{\R^m}} +O(R) + o_n(1).
\endaligned
\end{\equ}

Combining \eqref{identity}, \eqref{I2} and \eqref{I4},
we conclude that
\begin{\equ}\label{identity2}
\pa_k(H\circ\cs_a^{-1})(0)\int_{\R^m}
|\tilde\phi_n|^{p_n} d\vol_{\ig_{\R^m}}=O(R) + o_n(1)
\end{\equ}
as $n\to\infty$ and $R$ can be fixed arbitrarily small.
Since we already know from the blow-up analysis that
\[
\lim_{n\to\infty}\int_{\R^m}f^{\frac{m-1}2(2^*-p_n)}
|\tilde\phi_n|^{p_n}d\vol_{\ig_{\R^m}}
=\lim_{n\to\infty}\int_{S^m}
|\psi_n|^{p_n}d\vol_{\ig_{S^m}}>0,
\]
\eqref{identity2} gives us nothing but
$\pa_k(H\circ\cs_a^{-1})(0)\equiv0$. Notice that $k$ can be
varying from $1$ to $m$, we have
$\nabla(H\circ\cs_a^{-1})(0)=0$, i.e. $\nabla H(a)=0$
which completes the proof.
\end{proof}

\begin{Prop}\label{blow-up cor}
Suppose $\{\psi_n\}$ satisfies \eqref{equs-n} and
\eqref{key-assumption1} and does not contain any compact
subsequence.
Let $\{R_n\}$ be the associated radius found in Proposition \ref{blow-up prop}. Then
\[
\displaystyle\lim_{n\to\infty}R_n^{\frac{m-1}2(2^*-p_n)}=1.
\]
\end{Prop}

\begin{proof}
Let us recall the equation under stereographic projection \eqref{tilde phi-n} and consider the conformal change of $\tilde\phi_n$ as
\[
\tilde\phi_{n,R}(x)=R^{\frac{m-1}2}\tilde\phi_n(Rx) \quad \text{for } R>0.
\]
Then we have
\begin{\equ}\label{eq-tilde-n-r}
D_{\ig_{\R^m}}\tilde\phi_{n,R}=R^{\frac{m-1}2(2^*-p_n)}\widehat H_{n,R}|\tilde\phi_{n,R}|^{p_n-2}\tilde\phi_{n,R} \quad \text{on } \R^m
\end{\equ}
where, for ease of notations, we have denoted $\widehat H_{n,R}(x)=f^{\frac{m-1}2(2^*-p_n)}(Rx)\cdot
H\circ \cs_a^{-1}(Rx)$.

Let $\bt\in C_c^\infty(S^m)$ be the same cut-off function as in \eqref{identity}, we set
\[
\hat\phi_{n,R}(x)=\bt\circ\cs_a^{-1}(Rx)\cdot \tilde\phi_{n,R}(x).
\]
Then a direct calculation shows that
\begin{\equ}\label{id1}
\int_{\R^m}(D_{\ig_{\R^m}}\tilde\phi_{n,R},\hat\phi_{n,R})d\vol_{\ig_{\R^m}}
=\int_{S^m}(D\psi_n,\bt\psi_n) d\vol_{\ig_{S^m}}
\end{\equ}
and
\begin{\equ}\label{id2}
\int_{\R^m}(\bt\circ\cs_a^{-1})(Rx)\cdot\widehat H_{n,R}|\tilde\phi_{n,R}|^{p_n}d\vol_{\ig_{\R^m}}=R^{\frac{m-1}2(p_n-2^*)}
\int_{S^m}\bt H|\psi_n|^{p_n}d\vol_{\ig_{S^m}}.
\end{\equ}
Hence, take derivative with respect to $R$ in \eqref{id1}, we have
\begin{eqnarray}\label{idd1}
0&=&\frac{d}{dR}\Big|_{R=R_n}\int_{\R^m}(D_{\ig_{\R^m}}\tilde\phi_{n,R},\hat\phi_{n,R})d\vol_{\ig_{\R^m}} \nonumber \\
&=&2\real\int_{\R^m}\Big(D_{\ig_{\R^m}}\tilde\phi_{n,R_n},\frac{d}{dR}\Big|_{R=R_n}\hat\phi_{n,R}\Big)d\vol_{\ig_{\R^m}} \nonumber\\
& & \quad +\real\int_{\R^m}\Big( \frac{d}{dR}\Big|_{R=R_n}\tilde\phi_{n,R}, R_n\nabla(\bt\circ\cs_a^{-1})(R_nx)\cdot_{\ig_{\R^m}}\tilde\phi_{n,R_n}\Big)d\vol_{\ig_{\R^m}} \nonumber\\
& &\quad -\real\int_{\R^m}\Big( D_{\ig_{\R^m}}\tilde\phi_{n,R_n}, \frac{d}{dR}\Big|_{R=R_n}\big[(\bt\circ\cs_a^{-1})(Rx)\big]\tilde\phi_{n,R_n}\Big)d\vol_{\ig_{\R^m}}.
\end{eqnarray}
To evaluate the last two integrals above, we first notice that
\[
\frac{d}{dR}\Big|_{R=R_n}\tilde\phi_{n,R}(x)=\frac{m-1}2R_n^{\frac{m-3}2}\tilde\phi_n(R_nx)+R_n^{\frac{m+1}2}\nabla\tilde\phi_n(R_nx)\cdot x,
\]
and by the property of the hermitian product on $\mbs(S^m)$ (see the second axiom of the Dirac bundle) we have
\[
 \real\int_{\R^m}\big(\tilde\phi_n(R_nx),\nabla(\bt\circ\cs_a^{-1})(R_nx)\cdot_{\ig_{\R^m}}\tilde\phi_{n,R_n}  \big)d\vol_{\ig_{\R^m}} \equiv 0.
\]
Moreover, using the fact $\{\nabla\tilde\phi_n\}$ is uniformly bounded in $L^{\frac{2m}{m+1}}(\R^m,\mbs_m)$ and $\bt$ has a compact support, we obtain
\[
\aligned
&\Big|\int_{\R^m}\Big( R_n^{\frac{m+1}2}\nabla\tilde\phi_n(R_nx)\cdot x, R_n\nabla(\bt\circ\cs_a^{-1})(R_nx)\cdot_{\ig_{\R^m}}\tilde\phi_{n,R_n}\Big)d\vol_{\ig_{\R^m}}\Big| \\
&\quad =\Big|\int_{\R^m}\Big( \nabla\tilde\phi_n(x)\cdot x, \nabla(\bt\circ\cs_a^{-1})\cdot_{\ig_{\R^m}}\tilde\phi_n\Big)d\vol_{\ig_{\R^m}}\Big| \\
&\quad \leq C |\nabla\tilde\phi_n|_{\frac{2m}{m+1}}\cdot \Big( \int_{B_{3r}(a)\setminus B_{2r}(a)}|\psi_n|^{2^*} \Big)^{\frac1{2^*}}=o_n(1)
\endaligned
\]
as $n\to\infty$. Hence
\[
\Big|\int_{\R^m}\Big( \frac{d}{dR}\Big|_{R=R_n}\tilde\phi_{n,R}, R_n\nabla(\bt\circ\cs_a^{-1})(R_nx)\cdot_{\ig_{\R^m}}\tilde\phi_{n,R_n}\Big)d\vol_{\ig_{\R^m}}\Big|
=o_n(1)
\]
as $n\to\infty$. Analogously, it follows that
\[
\Big| \int_{\R^m}\Big( D_{\ig_{\R^m}}\tilde\phi_{n,R_n}, \frac{d}{dR}\Big|_{R=R_n}\big[(\bt\circ\cs_a^{-1})(Rx)\big]\tilde\phi_{n,R_n}\Big)d\vol_{\ig_{\R^m}} \Big|=o_n(1)
\]
as $n\to\infty$. And thus, from \eqref{idd1}, we find
\begin{\equ}\label{iddd1}
\real\int_{\R^m}\Big(D_{\ig_{\R^m}}\tilde\phi_{n,R_n},\frac{d}{dR}\Big|_{R=R_n}\hat\phi_{n,R}\Big)d\vol_{\ig_{\R^m}}=o_n(1)
\end{\equ}
as $n\to\infty$.

To proceed, we use \eqref{id2} to obtain
\begin{\equ}\label{idd2}
\aligned
&\frac{d}{dR}\Big|_{R=R_n}\int_{\R^m}(\bt\circ\cs_a^{-1})(Rx)\cdot\widehat H_{n,R}|\tilde\phi_{n,R}|^{p_n}d\vol_{\ig_{\R^m}}\\
&\quad =
\frac{m-1}2(p_n-2^*)R_n^{\frac{m-1}2(p_n-2^*)-1}\int_{S^m}\bt H|\psi_n|^{p_n}d\vol_{\ig_{S^m}}.
\endaligned
\end{\equ}
On the other hand,
\[
\aligned
&\frac{d}{dR}\Big|_{R=R_n}\int_{\R^m}(\bt\circ\cs_a^{-1})(Rx)\cdot\widehat H_{n,R}|\tilde\phi_{n,R}|^{p_n}d\vol_{\ig_{\R^m}}\\
&\quad =\int_{\R^m}\frac{d}{dR}\Big|_{R=R_n}\big[ (\bt\circ\cs_a^{-1})(Rx)\cdot\widehat H_{n,R} \big]\cdot|\tilde\phi_{n,R_n}|^{p_n}d\vol_{\ig_{\R^m}} \\
&\qquad  + p_n\real\int_{\R^m}(\bt\circ\cs_a^{-1})(R_nx)\cdot\widehat H_{n,R_n}|\tilde\phi_{n,R_n}|^{p_n-2}\Big(\tilde\phi_{n,R_,}, \frac{d}{dR}\Big|_{R=R_n}\hat\phi_{n,R}\Big)d\vol_{\ig_{\R^m}}\\
&\qquad -p_n\int_{\R^m}\frac{d}{dR}\Big|_{R=R_n}\big[(\bt\circ\cs_a^{-1})(Rx)\big]\cdot\widehat H_{n,R_n}|\tilde\phi_{n,R_n}|^{p_n}d\vol_{\ig_{\R^m}},
\endaligned
\]
and by using $f(x)=\frac2{1+|x|^2}$ and $\bt$ has compact support, we have
\[
\int_{\R^m}\frac{d}{dR}\Big|_{R=R_n}\big[ (\bt\circ\cs_a^{-1})(Rx)\cdot\widehat H_{n,R} \big]\cdot|\tilde\phi_{n,R_n}|^{p_n}d\vol_{\ig_{\R^m}}=O_n(1)
\]
and
\[
\aligned
&\Big|\int_{\R^m}\frac{d}{dR}\Big|_{R=R_n}\big[(\bt\circ\cs_a^{-1})(Rx)\big]\cdot\widehat H_{n,R_n}|\tilde\phi_{n,R_n}|^{p_n}d\vol_{\ig_{\R^m}}\Big| \\
&\quad \leq C \int_{B_{3r}(a)\setminus B_{2r}(a)}|\psi_n|^{p_n}d\vol_{\ig_{S^m}}=o_n(1)
\endaligned
\]
as $n\to\infty$. Thus, by virtue of \eqref{idd2}, we infer that
\begin{\equ}\label{iddd2}
\aligned
&\real\int_{\R^m}(\bt\circ\cs_a^{-1})(R_nx)\cdot\widehat H_{n,R_n}|\tilde\phi_{n,R_n}|^{p_n-2}\Big(\tilde\phi_{n,R_,}, \frac{d}{dR}\Big|_{R=R_n}\hat\phi_{n,R}\Big)d\vol_{\ig_{\R^m}}\\
&\quad =\frac{m-1}2(p_n-2^*)R_n^{\frac{m-1}2(p_n-2^*)-1}\int_{S^m}\bt H|\psi_n|^{p_n}d\vol_{\ig_{S^m}} + O_n(1)
\endaligned
\end{\equ}
as $n\to\infty$.

Combining \eqref{eq-tilde-n-r}, \eqref{iddd1}, \eqref{iddd2} and Lemma \ref{concentration parameter}, we
can conclude
\[
(2^*-p_n)R_n^{-1}\int_{S^m}\bt H|\psi_n|^{p_n}d\vol_{\ig_{S^m}}=O_n(1) \quad \text{as } n\to\infty.
\]
Since the blow-up phenomenon suggests $\lim_{n\to\infty}\int_{S^m}\bt H|\psi_n|^{p_n}d\vol_{\ig_{S^m}}>0$, we find
$2^*-p_n=O(R_n)$ as $n\to\infty$. Therefore
\[
\lim_{n\to\infty}R_n^{\frac{m-1}2(2^*-p_n)}=\lim_{n\to\infty}e^{O(1)R_n\ln R_n}=1.
\]
\end{proof}

\textit{\textbf{Acknowledgements}}. This paper was done while the author was visiting Gie\ss en Universit\"at (Germany)
as a Humboldt fellow.
The author would like to thank Prof.
Thomas Bartsch for his constant availability and
helpful discussions.

\vspace{2mm}
{\sc Tian Xu\\
 Center for Applied Mathematics, Tianjin University\\
 Tianjin, 300072, China}\\
 xutian@amss.ac.cn


\begin{thebibliography}{SK}

\normalsize
\baselineskip=17pt

\bibitem{Adams}R. Adams,
Sobolev Spaces,
Academic Press, New York, (1975).

\bibitem{Ammann}B. Ammann,
A variational problem in conformal spin geometry,
Habilitationsschift, Universit\"at Hamburg, (2003).

\bibitem{Ammann2009}B. Ammann,
The smallest Dirac eigenvalue in a spin-conformal class and cmc immersions,
Comm. Anal. Geom. 17 (2009), no. 3, 429-479.

\bibitem{AGHM}B. Ammann, J.-F. Grossjean, E. Humbert, B. Morel,
A spinorial analogue of Aubin's inequality,
Math. Z. 260 (2008), 127-151.

\bibitem{AHA}B. Ammann, E. Humbert, M. Ould. Ahmedou,
An obstruction for the mean curvature of a conformal immersion $S^n\to\R^{n+1}$,
Proc. Amer. Math. Soc. 135 (2007), no. 2, 489-493.

\bibitem{AHM}B. Ammann, E. Humbert, B. Morel,
Mass endomorphism and spinorial Yamabe type problems on conformally flat manifolds,
Comm. Anal. Geom. 14 (2006), no. 1, 163-182.

\bibitem{Aubin}T. Aubin, \'Equations diff\'erentielles non lin\'eaires et probl\`eme de Yamabe concernant la courbure scalaire,
J. Math. Pures Appl. 55 (1976), 269-296.

\bibitem{BBMW}B. Booss-Bavnbek, M. Marcolli, B.L. Wang,
Weak UCP and perturbed monopole equations,
Int. J. Math. 13, 987 (2002).

\bibitem{BG}J.-P. Bourguignon, P. Gauduchon,
Spineurs, op\'{e}rateurs de Dirac et variations de m\'{e}triques, Comm. Math. Phys. 144 (1992), no. 3, 581-599.

\bibitem{BahriC}A. Bahri, J.M. Coron,
The scalar curvature problem on the standard three-dimensional sphere,
J. Funct. Anal. 95 (1991), 106-172.

\bibitem{BC}H. Brezis, J.M. Coron,
Convergence of solutions of $H$-systems or how to blow bubbles,
Arch. Ration. Mech. Anal. 89 (1985), 21-56.

\bibitem{CGY}S.Y. Chang, M. Gursky, P. Yang,
The scalar curvature equation on 2- and 3-spheres,
Calc. Var. Partial Differential Equations 1 (1993),
205-229.

\bibitem{CY}S.Y. Chang, P. Yang,
A perturbation result in prescribing scalar curvature on $S^n$,
Duke Math. J. 64 (1991), 27-69

\bibitem{Friedrich} T. Friedrich,
Dirac Operators in Riemannian Geometry,
Grad. Stud. Math., vol 25, Amer. Math. Soc., Providence  (2000).

\bibitem{Ginoux} N. Ginoux,
The Dirac Spectrum,
Lecture Notes in Mathematics, vol. 1976. Springer, Berlin
(2009).

\bibitem{Hij86} O. Hijazi,
A conformal lower bound for the smallest eigenvalue of the Dirac operator and Killing spinors,
Comm. Math. Phys. 104 (1986), 151-162.

\bibitem{Hit74}N. Hitchin,
Harmonic spinors,
Adv. Math. 14 (1974), 1-55.

\bibitem{Isobe-JFA}T. Isobe,
Nonlinear Dirac equations with critical nonlinearities on compact Spin manifolds,
J. Funct. Anal. 260 (2011), no. 1, 253-307.

\bibitem{Isobe-MathAnn}T. Isobe,
A perturbation method for spinorial Yamabe type equations on $S^m$ and its application,
Math. Ann. 355 (2013), no. 4, 1255-1299.

\bibitem{KS}R. Kusner, N. Schmitt,
Representation of surfaces in space,
arXiv:dg-ga/9610005, (1996).

\bibitem{Lawson}H.B. Lawson, M.L. Michelson,
Spin Geometry,
Princeton University Press  (1989).

\bibitem{LY}P. Li, S.T. Yau,
A new conformal invariant and its applications to the Willmore conjecture and the first eigenvalue on compact surfaces,
Invent. Math. 69 (1982), 269-291.

\bibitem{Lions1}P.L. Lions,
The concentration-compactness principle in the calculus of variations. The limit case. Part 1,
Rev. Mat. Iberoamericana 1 (1) (1985), 145-201.

\bibitem{Lions2}P.L. Lions,
The concentration-compactness principle in the calculus of variations. The limit case. Part 2,
Rev. Mat. Iberoamericana 1 (2) (1985), 45-121.

\bibitem{Raulot}S. Raulot,
A Sobolev-like inequality for the Dirac operator,
J. Funct. Anal. 26 (2009), 1588-1617.

\bibitem{Sulanke} S. Sulanke,
Die Berechnung des Spektrums des Quadrates des Dirac-Operators auf der Sph\"{a}re,
Doktorarbeit. Humboldt-Universit\"at zu Berlin, Berlin, 1979.

\bibitem{Taimanov1}I.A. Taimanov,
Surfaces of revolution in terms of solitons,
Ann. Global Anal. Geom. 15 (1997), no. 5, 419-435.

\bibitem{Taimanov2}I.A. Taimanov,
Modified Novikov-Veselov equation and differential geometry of surfaces,
Amer. Math. Soc. Transl. Ser. 2, 179 (1997) 133-151.

\bibitem{Taimanov3}I.A. Taimanov,
The Weierstrass representation of closed surfaces in $\R^3$,
Funct. Anal. Appl. 32 (1998), no. 4, 258-267.

\bibitem{Yamabe}H. Yamabe,
On the deformation of Riemannian structures on compact manifolds, Osaka Math. J. 12 (1960) 21-37.

\end{thebibliography}
\end{document}